\documentclass{amsart}

\usepackage[all]{xy}
\usepackage{hyperref}
\usepackage{amssymb}
\usepackage{mathdots}
\usepackage{mathtools}
\usepackage{verbatim}

\usepackage[breakable,skins]{tcolorbox}
\newtcolorbox{tbox}[1][]{%
    breakable,
    enhanced,
    colframe=blue,
    coltitle=white,
    #1
}

\usepackage{tikz,tkz-euclide}

\hypersetup{
    colorlinks,    citecolor=blue,    filecolor=blue,    linkcolor=blue,    urlcolor=blue
}

\usepackage{booktabs}
\usepackage{mathrsfs}
\usepackage{todonotes}
\usepackage{bbm}
\usepackage[lite,bibtex-style]{amsrefs}
\usepackage{longtable}

\usepackage{MnSymbol}

\newtheorem{introthm}{Theorem}

\newtheorem{theorem}{Theorem}[section]
\newtheorem{lemma}[theorem]{Lemma}
\newtheorem{proposition}[theorem]{Proposition}
\newtheorem{corollary}[theorem]{Corollary}
\newtheorem{conjecture}[theorem]{Conjecture}
\theoremstyle{definition}

\newtheorem{definition}[theorem]{Definition}
\newtheorem{example}[theorem]{Example}

\newtheorem{construction}[theorem]{Construction}

\newtheorem{remark}[theorem]{Remark}
\theoremstyle{remark}

\numberwithin{equation}{section}

\title{On blowing up minimal 
toric surfaces}
\author[A.~Laface]{Antonio Laface}
\address{
Departamento de Matem\'atica,
Universidad de Concepci\'on,
Casilla 160-C,
Concepci\'on, Chile}
\email{alaface@udec.cl}

\author[L.~Ugaglia]{Luca Ugaglia}
\address{
Dipartimento di Matematica e Informatica,
Universit\`a degli studi di Palermo,
Via Archirafi 34,
90123 Palermo, Italy}
\email{luca.ugaglia@unipa.it}

\subjclass[2010]{Primary 14M25; Secondary 14C20}
\keywords{Toric surfaces, Cox rings}
\thanks{Both authors have been 
partially supported by Proyecto
FONDECYT Regular 
n.~1230287 and by 
``Piano strategico per il miglioramento della qualit\`a della ricerca e dei risultati della VQR 2020-2024 - Misura A''
of the University of Palermo.
The second author is member of 
INdAM - GNSAGA}

\date{\today}

\begin{document}

\maketitle

\begin{abstract}
We prove that the Cox ring of the blowing-up
of a minimal toric surface of Picard rank two
is finitely generated. As part of our proof 
of this result we provide a necessary and 
sufficient condition
for finite generation of Cox rings of 
normal projective $\mathbb Q$-factorial
surfaces.
\end{abstract}

\section*{Introduction}

Given a complete toric variety $\mathbb P$
and a general point $e\in\mathbb P$,
it is an open problem to decide when the
Cox ring of the blowing-up ${\rm Bl}_e\mathbb P$
is finitely generated.
For example when $\mathbb P$ is a (fake) weighted 
projective plane, the problem of deciding non
finite generation of the Cox ring can be 
related to the Nagata conjecture for plane 
curves~\cite{ck}. More recently examples of
surfaces ${\rm Bl}_e\mathbb P$ with non-polyhedral
pseudoeffective cone have been given in~\cite{cltu}
and used to show that the pseudoeffective cone of
$\overline M_{0,n}$ is not polyhedral for $n\geq 10$.
In this paper we focus on minimal toric surfaces,
that is toric surfaces which do not contain
curves with negative self-intersection.
These are quotients, by a cyclic subgroup
of the big torus, of either $\mathbb P^2$
or of $\mathbb P^1\times\mathbb P^1$.
In the first case (fake weighted projective 
planes) there are
several known such surfaces $\mathbb P$
such that the Cox ring of 
${\rm Bl}_e\mathbb P$ is not finitely 
generated~\cites{ggk2021,ggk2019,ggkneg,gk,ck}.
Recently in~\cite{ggkopen} the authors provided 
examples where the effective cone of 
${\rm Bl}_e\mathbb P$ is open on one side.
Here our contribution is a necessary and sufficient
condition for the Cox ring
to be generated in multiplicity one
(see Definition~\ref{def:mul}),
generalizing~\cite{hkl2}*{Thm. 1.1}.
Our main result concerns the second case 
where we prove the following.
\begin{introthm}
\label{thm:main}
Let $\mathbb P$ be a minimal toric surface
of Picard rank two. Then the Cox ring of 
${\rm Bl}_e\mathbb P$ is finitely generated.
\end{introthm}
Our strategy for proving this result is
the following. First of all we characterize 
the strict transforms of
one-parameter subgroups that are negative curves.
Then we show that these classes, together 
with the exceptional divisor and the pullback
of two classes of self-intersection $0$ 
generate the effective cone. 
Finally we provide a necessary and sufficient condition for the 
finite generation of the Cox ring of 
surfaces with a polyhedral effective cone
(Proposition~\ref{lem:mah}).
In Theorem~\ref{thm:eff} we show that 
${\rm Eff}({\rm Bl}_e\mathbb P)$ is generated
in multiplicity one but we don't know if the
same is true for the Cox ring. 
In Theorem~\ref{thm:gen} we provide
a necessary condition for the Cox ring of 
the blowing-up of a toric surface to be 
generated in multiplicity one.
After performing an experimental analysis 
with the computer algebra system 
Magma~\cite{mag}
we conjecture that the condition is also 
sufficient (Conjecture~\ref{con:m1})
and we prove it when the lattice
ideal of the toric surface is Cohen-Macaulay 
(Theorem~\ref{thm:CM}).
Finally, 
assuming Conjecture~\ref{con:m1} to hold,
in Corollary~\ref{cor:m1} we prove that
if $\mathbb P$ is a minimal toric surface
of Picard rank two
then the Cox ring of 
${\rm Bl}_e\mathbb P$ is generated
in multiplicity $1$.

The paper is structured as follows. 
In Section~\ref{sec:pre} we introduce
two criteria that guarantee the 
finite generation of the effective cone and 
the Cox ring 
of a normal projective $\mathbb Q$-factorial
surface. In Section~\ref{sec:tor} we first
recall some definitions and results about
abelian monoids, lattice ideals and toric 
varieties $\mathbb P$ associated
to a fan $\Sigma$ in a free abelian group $M$.
We then focus on the case of toric surfaces,
showing first that the classes of 
of one-parameter subgroups
induce a fan structure on the nef cone 
${\rm Nef}(\mathbb P)$. Then 
in Theorem~\ref{thm:ts} we prove inclusions 
between finite subsets of 
$N:={\rm Hom}(M,{\mathbb Z})$ 
corresponding respectively to: negative 
curves on the blow-up ${\rm Bl}_e\mathbb P$, 
width directions
of Riemann-Roch polytopes of ample classes 
on $\mathbb P$, and a minimal binomial basis 
of the lattice ideal of $\mathbb P$.
In Section~\ref{sec:li} we prove a necessary
condition for generation of the Cox ring 
in multiplicity one, in terms of the 
lattice ideal of the toric surface and
we show that it is also sufficient when
the lattice ideal is Cohen-Macaulay
(Theorem~\ref{thm:CM}).
In Section~\ref{sec:min} we 
restrict to the case of minimal 
toric surfaces of Picard rank $2$: 
in Theorem~\ref{thm:eff} we show that 
the effective cone ${\rm Eff}({\rm Bl}_e
\mathbb P)$ is generated in multiplicity
$1$, and then we prove
Theorem~\ref{thm:main}. Finally,
in Section~\ref{sec:ex} we present
a library of Magma~\cite{mag} functions and 
we discuss in detail one example, 
providing evidence for Conjecture~\ref{con:m1}.

\vspace{3mm}

{\bf Acknowledgements.}
We would like to thank Ana-Maria Castravet 
for some useful discussions about the topics of
the paper.


\section{Preliminaries}
\label{sec:pre}
In this section we prove a criterion
for finite generation of the effective 
cone of a normal projective $\mathbb Q$-factorial
surface and a criterion for finite generation of 
the Cox ring.
Both results will be used in 
Section~\ref{sec:min} to prove our main result, 
i.e. Theorem~\ref{thm:main}.
In what follows we will adopt
the following notation: given
a tuple of Weil divisors classes
$\mathcal S := (D_1,\dots,D_n)$
on a projective variety $X$ with
finitely generated divisor class 
group ${\rm Cl}(X)$ we will denote
by ${\rm Cone}(\mathcal S)$ the
cone of ${\rm Cl}_{\mathbb R}(X)$
generated by the classes $D_1,
\dots,D_n$.
We will also denote by 
\[
 Q_X\subseteq {\rm Cl}_{\mathbb R}(X)
\]
the positive light cone of $X$, i.e.
the set of classes $D\in {\rm Cl}_{\mathbb R}(X)$
such that $D^2\geq 0$ and $D\cdot H\geq 0$,
for some ample class $H$.
\begin{proposition}
\label{lem:eff}
Let $X$ be a normal projective 
$\mathbb Q$-factorial surface with 
Picard rank $3$. If there exists
a tuple $\mathcal S := (D_1,\dots,D_n)$
of classes of irreducible
curves on $X$ such that 
$\dim({\rm Cone}({\mathcal S})) = 3$
and the
intersection form is negative semi-definite 
on ${\rm Cone}(D_n,D_1)$ and 
${\rm Cone}(D_i,D_{i+1})$, 
for $1\leq i\leq n-1$, 
then ${\rm Eff}(X) = {\rm Cone}(\mathcal S)$.
\end{proposition}
\begin{proof}

Since $D_i^2 \leq 0$ 
for any $i$, there exists a
subset $I\subseteq\{1,\dots, n\}$
such that $|I| \geq 3$ and
the extremal rays of the cone
${\rm Cone}(\mathcal S)$ are
$\{D_i \,:\, i\in I\}$.
Moreover the intersection form
is negative semi-definite on 
the facets of this cone, so that
either $Q_X\subseteq
{\rm Cone}(\mathcal S)$, or
${\rm Cone}(\mathcal S)$ does 
not intersect the interior of
$Q_X$. We claim that the latter 
is not possible. Indeed, given 
an extremal ray $D_i$ of 
${\rm Cone}({\mathcal S})$, by
the Hodge index theorem we have
that $D_i^{\perp}$ intersects
the interior of $Q_X$. This implies
that $D_i\cdot D_j < 0$
for any other extremal 
ray $D_j$ of 
${\rm Cone}({\mathcal S})$, a contradiction.

Therefore $Q_X\subseteq
{\rm Cone}(\mathcal S)$,
and we conclude by means
of~\cite{lu}*{Prop.~1.2}.
\end{proof}

\begin{definition}
\label{def:main}
Let $X$ be a normal projective $\mathbb Q$-factorial
surface with finitely generated divisor class 
group ${\rm Cl}(X)$.
We say that an unordered tuple
$\mathcal S$ consisting of elements of ${\rm Cl}(X)$
is a {\em pseudogenerating tuple}
if it consists of classes of irreducible
curves such that the following hold:
\begin{enumerate}
\item
$Q_X\subseteq{\rm Cone}(\mathcal S)$;
\item
for any $D\in {\rm Cl}(X)$ generating
an extremal ray of ${\rm Nef}(X)$ we have
\[
 D\in 
 \bigcap_{C \in \mathcal S\cap D^\perp}
 {\rm Cone}(\mathcal S \setminus C).
\]
\end{enumerate}
\end{definition}

\begin{proposition}
\label{lem:mah}
Let $X$ be a normal projective 
$\mathbb Q$-factorial surface 
with finitely generated divisor
class group ${\rm Cl}(X)$ of rank
at least $2$. 
Then the following are equivalent:
\begin{enumerate}
\item
$X$ admits a pseudogenerating tuple;
\item
the Cox ring of $X$ is finitely generated.
\end{enumerate}
\end{proposition}
\begin{proof}
We prove $(i)\Rightarrow (ii)$.
By Hu and Keel criterion (see~\cite{adhl})
the Cox ring of $X$ is finitely generated 
if and only if ${\rm Nef}(X)$ is generated
by finitely many semiample classes.
Let $\mathcal S\subseteq {\rm Cl}(X)$
be a pseudogenerating tuple.
The first assumption on $\mathcal S$
implies, by~\cite{lu}*{Prop.~1.2}
\footnote{The argument given for Picard 
rank at least three works for Picard 
rank two as well.}, that ${\rm Eff}(X) = {\rm Cone}(\mathcal S)$,
and in particular it
is finitely generated.
Thus ${\rm Nef}(X)$ is generated
by finitely many classes, being the dual
of ${\rm Eff}(X)$.
It remains to show that any extremal ray
of ${\rm Nef}(X)$ is generated by a 
semiample class. Let $D$ be a 
class which generates an extremal ray of 
${\rm Nef}(X)$. The inclusion 
${\rm Nef}(X)\subseteq \overline{\rm Eff}(X)
= {\rm Eff}(X) = {\rm Cone}(\mathcal S)$ 
implies that a multiple of $D$ is a 
non-negative linear combination of elements of 
$\mathcal S$. Thus, up to replace $D$ with 
this multiple,
we can write 
$D = \sum_{[C]\in S}m_C[C]$,
where $m_C\in\mathbb Z_{\geq 0}$ for any $m_C$,
and each $C\subseteq X$ is an irreducible curve.
It follows that
\[
 {\bf B}(D)\subseteq 
 \bigcup_{[C]\in \mathcal S}C,
\]
that is the stable base locus of $D$ is
contained in the union of all curves of 
$\mathcal S$.
If $D$ is big, then ${\bf B}(D)\subseteq 
D^\perp$ by~\cite{na}*{Thm. 1.1} and the
second assumption on $\mathcal S$ implies
that no curve of $\mathcal S\cap D^\perp$
can be in ${\bf B}(D)$.
If $D$ is not big, then $0 = D^2 = D\cdot 
\sum_{[C]\in S}m_CC = \sum_{[C]\in S}m_C(D\cdot C)$
implies that ${\bf B}(D)\subseteq D^\perp$
and again we conclude that no curve, whose class 
is in $\mathcal S\cap D^\perp$, 
can be in ${\bf B}(D)$.
Finally we deduce that ${\bf B}(D) = \emptyset$,
that is $D$ is semiample, by~\cite{la}*{Rem. 2.1.32}.

We prove $(ii)\Rightarrow (i)$.
Let $\mathcal S$ be the tuple of degrees
of a minimal generating subset of the Cox ring.
Observe that all the element of $\mathcal S$
are classes of irreducible curves of $X$.
The first condition of Definition~\ref{def:main}
holds because $Q_X\subseteq \overline{\rm Eff}(X)
= {\rm Eff}(X) = {\rm Cone}(\mathcal S)$,
where the inclusion is by Riemann-Roch,
the second equality is by hypothesis and
the first equality is because ${\rm Eff}(X)$
is polyhedral.
The second condition holds because
\[
 D\in 
 {\rm Nef}(X)
 = {\rm Mov}(X)
 =
 \bigcap_{C \in \mathcal S}
 {\rm Cone}(\mathcal S \setminus C)
 \subseteq
 \bigcap_{C \in \mathcal S\cap D^\perp}
 {\rm Cone}(\mathcal S \setminus C),
\]
where the first equality is by Hu and Keel
characterization of finite generation of 
the Cox ring and the second equality is by 
\cite{adhl}*{Prop. 3.3.2.3}.
\end{proof}

\begin{remark}
\label{rem:gen}
The above proposition provides a criterion
for finite generation of the Cox ring but
it does not give its generators. 
For instance, let us consider $X$ 
to be a del Pezzo surface
of degree $1$.
Since $X$ is a Mori dream space, the rays of
${\rm Eff}(X)$ are generated by 
classes of negative curves, which can only 
be $(-1)$-curves.
Therefore we can write ${\rm Eff}(X) = {\rm Cone}(\mathcal S)\supseteq Q_X$, where
$\mathcal S := \{[C] \, : \, 
C \text { is a $(-1)$-curve}\}$.
The extremal rays of the nef cone of $X$
are generated by either pullbacks of lines
(contracting $8$ disjoint $(-1)$-curves)
or by conic bundles. Thus $\mathcal S$
is a pseudogenerating tuple.
On the other hand the Cox ring 
of $X$ is generated by homogeneous elements
whose degrees are either classes of curves
in $\mathcal S$ or the anticanonical class 
$-K_X$.
\end{remark}

\begin{corollary}
\label{cor:4C}
Let $X$ be a normal projective 
$\mathbb Q$-factorial surface 
with divisor class group ${\rm Cl}(X)$
of rank $2$. Then the following are 
equivalent:
\begin{enumerate}
\item
the Cox ring of $X$ is finitely generated;
\item
there are irreducible curves $C_1,D_1,
C_2,D_2\subseteq X$ such that 
$C_i\cdot D_i = 0$ for $i=1,2$.
\end{enumerate}
\end{corollary}
\begin{proof}
We prove (i) $\Rightarrow$ (ii).
If the Cox ring of $X$ is finitely
generated then ${\rm Eff}(X)$ is
polyhedral, so that its two extremal 
rays are generated by classes of 
irreducible curves $C_1$ and $C_2$.
The extremal rays of ${\rm Nef}(X)$
lie on $C_i^\perp$. 
If $R_i\in C_i^\perp$
is a nef divisor, then it is semiample
by the hypothesis. Thus, up to replace
$R_i$ with a positive multiple, we can
assume that the linear system $|R_i|$
defines a morphism $\varphi\colon 
X\to Y$ with connected fibers 
onto either a curve or a surface
\cite{la}*{Thm. 2.1.27}.
An irreducible component $D_i$ 
of a general fiber of $\varphi$ 
is disjoint from $C_i$.

We prove (ii) $\Rightarrow$ (i).
Let $\mathcal S := (C_1,D_1,C_2,D_2)$.
The equation $C_i\cdot D_i=0$ implies 
that for any $i$, either $C_i$ or $D_i$ 
has self-intersection $\leq 0$. 
Thus the first hypothesis on $\mathcal S$
of Definition~\ref{def:main} is satisfied.
If $C_i^2<0$, then $[D_i]\in 
{\rm Cone}(\{D_i\})\subseteq
{\rm Cone}(\mathcal S\setminus C_i)$.
If $C_i^2=0$, then $D_i^2=0$ and the two classes $[C_i]$, $[D_i]$ span the same
extremal ray $R$ of ${\rm Nef}(X)$.
In this case we have $[C_i],[D_i]\in 
R\subseteq 
{\rm Cone}(\mathcal S\setminus C_i)\cap
{\rm Cone}(\mathcal S\setminus D_i)$.
Therefore $\mathcal S$ is a 
pseudogenerating tuple for $X$, and
we conclude by means of Proposition~\ref{lem:mah}.
\end{proof}

\section{Toric surfaces}
\label{sec:tor}
In this section we are going to 
recall some definitions and 
to prove some facts about abelian
monoids, toric varieties, lattice
ideals and toric surfaces.
\subsection{Abelian monoids}
Recall that an {\em abelian monoid} 
is a pair $(S,+)$
consisting of a set $S$ and a binary 
operation 
$S\times S\to S$, 
which is commutative, associative and admits a neutral element $0\in S$.
The abelian monoid $S$ is {\em cancellative}
if $a+c = b+c$ implies $a=b$, where 
$a,b,c\in S$, while $S$ is {\em finitely 
generated} if there is a surjective 
homomorphism of monoids $\mathbb N^r\to S$.

\begin{definition}
Let $S$ be an abelian monoid.
An {\em ideal} $I\subseteq S$ of 
$S$ is a subset such that $i+s\in I$
for any $i\in I$ and $s\in S$.
Given a subset $U\subseteq S$, the ideal 
generated by $U$ is 
\[
 \langle U\rangle
 :=
 \{u+s\, :\, u\in U\text{ and }s\in S\}.
\]
A {\em basis} of an ideal 
$I\subseteq S$ is a subset $B\subseteq S$ 
such that $\langle B \rangle = I$.
One says that $B$ is a {\em minimal 
basis} for $I$ if $B$ is a basis but
no proper subset of $B$ is.
If the basis $B$ is finite, we say 
that the ideal $\langle B \rangle$ is
finitely generated.
\end{definition}
\begin{proposition}
\label{fin-bas}
Let $S$ be a finitely generated abelian
monoid. Then any minimal basis of an ideal
of $S$ is finite.
\end{proposition}
\begin{proof}
By hypothesis there is a surjective 
homomorphism of monoids 
$\pi\colon\mathbb N^r\to S$.
Let $I\subseteq S$ be an ideal.
If a minimal basis of 
$I$ were not finite then
there would be an infinite chain of
ideals $I_1\subsetneq I_2\dots
\subsetneq I_n\subsetneq\cdots\subsetneq I$,
where all the inclusions are strict.
This would give an infinite chain of
ideals $\pi^{-1}(I_1)\subsetneq 
\pi^{-1}(I_2)\dots
\subsetneq \pi^{-1}(I_n)
\subsetneq\cdots\subsetneq \pi^{-1}(I)$
where all the inclusions are strict.
A contradiction by~\cite{rg}*{Lem.~6.9}.
\end{proof}

We now provide an application of Proposition~\ref{fin-bas} to 
cones of divisor classes of algebraic
varieties.
Given a normal $\mathbb Q$-factorial  
algebraic variety $X$, denote by $C_X\subseteq
N_1(X)_{\mathbb R}$ the monoid of 
numerical classes of effective curves
of $X$.
The closure of the cone generated by 
$C_X$ is the Mori cone of $X$.

\begin{construction}
\label{con:fan}
Let $X$ be a normal $\mathbb Q$-factorial
algebraic variety such that the monoid 
$C_X$ is finitely generated. 
Given $x\in X$ denote by $I_x$ the
ideal of $C_X$ consisting of classes of 
curves which pass through $x$, and define
the function
\[
 \phi\,\colon\, {\rm Nef}(X)\to\mathbb R_{\geq 0},
 \qquad
 D\mapsto \min\{D\cdot\Gamma\, :\, 
 \Gamma\in I_x\}.
\]
Since $I_x$ admits a finite basis, the
function $\phi$ is the minimum of a 
finite number of linear functions.
It follows that $\phi$ is piece-wise 
linear and its domains of linearity 
are closed convex cones which form
a fan that we denote by $\Sigma_{\rm Nef}(X)$.
\end{construction}

\subsection{Toric Varieties}
Let $M$ be be a rank $n$ free abelian group and let $N := {\rm Hom}(M,\mathbb Z)$ be its dual.
We recall that a lattice polytope $\Delta\subseteq M\otimes_{\mathbb Z}\mathbb{Q}$
is the convex hull of a finite number of lattice points of $M$.

\begin{definition}
\label{rem:w}
Let ${\mathbb P}$ be a projective toric variety 
with lattice of one-parameter subgroups $N$
and let $e\in {\mathbb P}$ be a point in the big torus 
orbit.
Any $v\in N$ defines a curve 
$C_v\subseteq {\mathbb P}$ which is the 
closure of the image of the map $k^*
\to {\mathbb P}$, given by $t\mapsto t^v\cdot e$.
Let
\[
 {\rm OP}_{\mathbb P} :=\langle[C_v]\, :\,  v\in N\rangle
 \subseteq C_{\mathbb P},
\]
be the ideal of $C_{\mathbb P}$ generated by
the classes of one-parameter subgroups
through $e$. Observe that ${\rm OP}_{\mathbb P}\subseteq I_e$,
where $I_e$ is the ideal of $C_{\mathbb P}$
generated by classes of curves 
which contain $e$. Following 
Construction~\ref{con:fan} we define a 
piece-wise linear function
\[
 \phi^{\rm op}\, :\,  {\rm Nef}({\mathbb P})\to\mathbb R_{\geq 0},
 \qquad
 D\mapsto \min\{D\cdot\Gamma\, :\, 
 \Gamma\in {\rm OP}_{\mathbb P}\}
\]
whose domains of linearity
are convex cones which form a fan that we
denote by $\Sigma_{\rm nef}^{\rm op}({\mathbb P})$.
The inclusion ${\rm OP}_{\mathbb P}\subseteq I_e$ implies
that $\Sigma_{\rm nef}({\mathbb P})$ is a refinement
of $\Sigma_{\rm nef}^{\rm op}({\mathbb P})$.


\end{definition}

\begin{proposition}
If ${\mathbb P}$ is a projective toric surface 
then ${\rm OP}_{\mathbb P} = I_e$ and thus
$\Sigma_{\rm nef}({\mathbb P}) = \Sigma_{\rm nef}^{\rm op}({\mathbb P})$.
\end{proposition}
\begin{proof}
Let $p\,\colon \widehat {\mathbb P}\to {\mathbb P}$ 
be the characteristic space of ${\mathbb P}$ and 
let $\overline {\mathbb P}$ be the total space 
whose coordinate ring is the Cox ring
$\mathcal R({\mathbb P})$. Any curve of 
${\mathbb P}$ which passes through $e$ defines 
an element of the ideal sheaf of $e$ and thus 
an element of the ideal 
$I_{\mathbb P}
\subseteq \mathcal R({\mathbb P})$ of $p^{-1}(e)$
(see Subsection~\ref{subs:tor}). 
Since $p^{-1}(e)$
is a quasitorus orbit, it follows
that its defining ideal $I_{\mathbb P}$
is generated by binomials.
Each such binomial corresponds to 
a one-parameter subgroup of $X$.
The statement follows.
\end{proof}

\begin{definition}
\label{def:lw}
Given a {\em lattice direction}, i.e. a non zero primitive vector 
$v\in N$, let us denote respectively by $\min\langle\Delta,v\rangle$ and 
$\max\langle\Delta,v\rangle$ the minimum and the maximum of 
$\langle m,v\rangle$ for $m\in \Delta$. The {\em lattice width of} $\Delta$ {\em in the direction} $v$ can be defined as
\[
 {\rm lw}_v(\Delta) := \max\langle\Delta,v\rangle  - 
  \min\langle\Delta,v\rangle.
\]
The {\em lattice width of} $\Delta$ is defined as 
\[
 {\rm lw}(\Delta) := \min\{{\rm lw}_v(\Delta)\, :\,  v\in N\}.
\]
If $v\in N$ is such that ${\rm lw}_v(\Delta) = {\rm lw}(\Delta)$,
we say that $\pm v$ is a {\em width direction} for $\Delta$ and 
we denote by $\rm wd(\Delta) \subseteq N$ 
the set of width directions for $\Delta$. 
\end{definition}

Given a lattice polytope $\Delta$ denote
by $\Delta_{\pm} := \Delta -\Delta$, 
which is 
the Minkowski sum of $\Delta$ with $-\Delta$,
and by $\Delta_{\pm}^* := \{v\in 
N_\mathbb Q\, :\, 
\max\langle \Delta_{\pm},v\rangle\leq 1\}$ the polar polytope.

\begin{proposition}
\label{prop:w}
Let $\Delta\subseteq M_\mathbb Q$ be a 
lattice polytope. Then the following hold:
\begin{enumerate}
\item
${\rm lw}(\Delta) = \min\{m\in\mathbb Z_{>0}\, :\, |m\Delta_{\pm}^*\cap N|>1\}$. 
\item
${\rm wd}(\Delta) = {\rm lw}(\Delta)\cdot\Delta_{\pm}^*\cap N
\setminus \{0_N\}$.
\end{enumerate}
\end{proposition}
\begin{proof}
We begin by observing that, given $v\in N$,
${\rm lw}_v(\Delta) = \max\langle \Delta_{\pm},v\rangle$. Indeed both
sides of the equality have the form 
$\langle u,v\rangle-\langle u',v\rangle
= \langle u-u',v\rangle$,
where $u,u'\in\Delta$. It follows that
$\max\langle \Delta_{\pm},v\rangle\geq 
{\rm lw}(\Delta)$ for any $v\in N\setminus\{0_N\}$. Equivalently the minimum $m\in\mathbb Z_{>0}$ such that $m\Delta_{\pm}^*\cap N
= \{v\in N\, :\, \max \langle \Delta_{\pm},v\rangle\leq m\}$
contains a non-zero class is ${\rm lw}(\Delta)$. This proves both statements.
\end{proof}

\begin{remark}
 \label{rem:boundary}
 The lattice points of ${\rm lw}(\Delta)\cdot\Delta_{\pm}^*\cap N
 \setminus \{0_N\}$ lie on the
 boundary of ${\rm lw}
 (\Delta)\cdot\Delta_{\pm}^*$.
 Moreover the above proposition implies
 that the set $\rm wd(\Delta)$ is finite. In~\cite{DMN} it is
 shown that $|\rm wd(\Delta)| \leq 3^n-1$,
 where $n = \dim(\Delta)$,
 and equality holds if and only if 
 $\Delta$ is equivalent to a multiple of
 the {\em standard lattice cube}, i.e. 
 the polytope with vertices $\pm e_1,\dots,
 \pm e_n$. In particular, if $n = 2$ and 
 $\Delta$ is not equivalent to a multiple 
 of the standard lattice cube, then 
 $|{\rm wd}(\Delta)| \leq 6$. 
\end{remark}

\begin{corollary}
\label{cor:3wd}
If a lattice polygon $\Delta$ has exactly $6$ width directions, say $\pm u,\, \pm v,\, 
\pm w\, \in N$,
then, after possibly changing a sign, $ u + v + w = 0_N$.
\end{corollary}
\begin{proof}
If we denote by $m > 1$ the width of $\Delta$, by Proposition~\ref{prop:w}
the width directions ${\rm wd}(\Delta)$ are all the lattice points on the 
boundary of $m\Delta_{\pm}^*$.
The lattice polygon $\Delta'$ 
spanned by the lattice points of
$m\Delta_{\pm}^*$ 
is centrally symmetric and 
it contains exactly $6$ lattice points on its boundary. 
If we consider the list of 
lattice polygons with $1$ interior
lattice point given in~\cite{Ra}*{Thm.~5}, we see that the hexagon
with vertices $(\pm 1,0),\,(0,\pm 1),\, (\pm 1,\pm 1)$ is the only 
one which is centrally symmetric and
which contains $6$ lattice points on
its boundary. The claim follows.
\end{proof}


\begin{definition} Let $\Sigma$ be a fan and $\mathbb P := \mathbb P_\Sigma$
the corresponding toric variety.

\noindent We define the {\em set of width directions} of $\mathbb P$ as
\[
 {\rm wd}(\mathbb P) 
 := 
 \bigcup_{\substack{D\text{ ample} \\
\text{invariant}}}
 {\rm wd}(\Delta_D)
 \subseteq N.
\]
\end{definition}

\begin{theorem}
Let $\mathbb P$ be a $\mathbb Q$-factorial 
projective toric variety. 
Then the following hold:
\begin{enumerate}
\item
${\rm wd}(\mathbb P)$ is finite;
\item
${\rm wd}$ is constant on the 
relative interior of cones of
$\Sigma_{\rm Nef}^{\rm op}(\mathbb P)$.
\end{enumerate}
\end{theorem}
\begin{proof}
We prove $(i)$.
Observe that the monoid $C_{\mathbb P}$ 
is finitely generated since
a generating set consists of classes of
torus invariant irreducible curves.
The fiber of the map 
\[
 \varphi\colon\, {\rm wd(\mathbb P)}\to 
 {\rm OP}_{\mathbb P},
 \qquad
 v\mapsto [C_v]
\]
over $[C_v]$ consists of $\{v,-v\}$, in particular the map has finite fibers.
Let $B\subseteq {\rm OP}_{\mathbb P}$ be a minimal
basis.
Since by Proposition~\ref{fin-bas} 
$B$ is finite, in order to conclude it is
enough to show that the image of
$\varphi$ is contained in $B$.
This is equivalent to prove that 
given an ample invariant divisor $D$ 
of $\mathbb P$, if $v\in N$ is a width
direction for the corresponding lattice polytope $\Delta_D$, then $[C_v]\in B$.
Let us suppose on the contrary that 
$[C_v]\notin B$. By definition of minimal
basis, there exists $[C_u]\in B$ such that
$[C_v] = [C_u] + s$, where $s\in C_{\mathbb P}$ is 
the class of an effective curve. This
would imply
\[
{\rm lw}_v(\Delta_D) = 
D\cdot C_v > D\cdot C_u = {\rm lw}_u(\Delta_D),
\]
which contradicts the fact that $v$ is a
width direction for $\Delta_D$.

We prove $(ii)$. It suffices to
observe that for all the divisor 
classes in the relative
interior of a cone of $\Sigma_{\rm Nef}^{\rm op}(\mathbb P)$
the function $\phi^{\rm op}$ attains
its minimum on the same elements of $B$.

\end{proof}

\subsection{Toric surfaces}
\label{subs:tor}
\begin{definition}
Let $\mathbb P := \mathbb P_\Sigma$ be a projective 
toric variety, with fan 
$\Sigma\subseteq N_{\mathbb Q}$.
Given a cone $\sigma\in\Sigma$ we
denote by ${\rm Hbs}(\sigma)$ its
Hilbert basis and define 
\[
 {\rm Hbs}(\mathbb P) := \bigcup_{\sigma\in\Sigma}
 \{v\in N\, :\, v\in {\rm Hbs}(\sigma)\},
 \qquad
 {\rm Hbs}(\mathbb P)^\pm := 
 {\rm Hbs}(\mathbb P)\cap -{\rm Hbs}(\mathbb P).
\]
\end{definition}

\begin{definition}
Let $\mathbb P := \mathbb P_\Sigma$ be a projective 
toric surface, with fan 
$\Sigma\subseteq N_{\mathbb Q}$
and let $\pi\colon {\rm Bl}_e\mathbb P\to 
\mathbb P$
be the blowing-up at a general 
point $e\in \mathbb P$. Given $v\in N$
we denote by $\tilde C_v\subseteq
{\rm Bl}_e\mathbb P$ the strict transform of 
the curve $C_v\subseteq \mathbb P$.
We define 
\[
 {\rm Neg}(\mathbb P) := \{v\in N\, 
 :\, \tilde C_v^2 < 0\}.
\]
Observe that since the multiplicity of $C_v$
at $e$ is $1$, 
we can also write
${\rm Neg}(\mathbb P) = \{v\in N\, 
 :\, C_v^2 < 1\}$.
\end{definition}

The next definition requires more
preliminaries to be introduced.
Given $m\in \mathbb Z^n$
define $m^+ := (\max(m_i,0)\, :\, 1\leq i\leq n)$
and $m^- := (-\min(m_i,0)\, :\, 1\leq i\leq n)$,
so that both $m^+$ and $m^-$ have 
non-negative entries and $m = m^+-m^-$.
We denote by $x^m$ the Laurent monomial
$\prod_ix_i^{m_i}\in k[x_1^{\pm 1},
\dots,x_n^{\pm 1}]$. Given a subgroup $L\subseteq\mathbb Z^n$ its {\em lattice 
ideal} is 
\[
 I_L
 :=
 \langle x^{m^+} - x^{m^-}\, :\, m\in L\rangle.
\]
It is known~\cite{PS}*{Thm.~3.7} that a minimal 
binomial basis for a lattice ideal is 
essentially unique (up to signs) if 
the lattice ideal is not a complete 
intersection.
Given a projective toric variety
$\mathbb P := \mathbb P_\Sigma$, with fan 
$\Sigma\subseteq N_{\mathbb Q}$,
let $v_1,\dots,v_r\in N$ be the primitive 
generators of the one-dimensional
cones of $\Sigma$. 
Denote by $P_\Sigma^*\,\colon\, \mathbb Z^r\to
\mathbb Z^n$ the homomorphism defined 
by $e_i\mapsto v_i$ for $1\leq i\leq r$
and let $L_\Sigma\subseteq\mathbb Z^r$ 
be the image of the dual homomorphism 
$P_\Sigma$. In what follows we
will denote simply by 
$I_{\mathbb P}$ the lattice
ideal of the lattice $L_{\Sigma}$.

\begin{definition}
Let $\mathbb P := \mathbb P_\Sigma$ be a projective 
toric surface, with fan 
$\Sigma\subseteq N_{\mathbb Q}$
and let $B_\Sigma$ be a minimal basis 
of the lattice ideal $I_{\mathbb P}$.
We define
\[
 {\rm Lib} (\mathbb P) := \{v\in N\, 
 :\, m\in v^\perp\text{ and }
 x^{m^+} - x^{m^-}\in B_\Sigma\}.
\]
\end{definition}

\begin{remark}
Observe that ${\rm Lib} (\mathbb P)$ is uniquely
defined only if the lattice ideal is
not a complete intersection. For example 
if $\mathbb P$ is the projective plane
this is not the case.
In general a basis of a lattice 
$L\subseteq \mathbb Z^r$ does not 
give a basis $B_L$ of the lattice ideal $I_{\mathbb P}$ 
because a saturation step with respect to the product $x_1\cdots x_r$ of all variables is 
needed~\cite{ms}*{Lem.~7.6}.
On the other hand if $\Phi\,\colon\, 
\mathbb Q^r\oplus \mathbb Q^r
\to \mathbb Q^r$ is the difference map
$(a,b)\mapsto a-b$, then one can show 
that $B_L$ is a subset of the Hilbert 
basis $\mathscr B$ of the cone 
\[
 \sigma_L 
 := 
 \Phi^{-1}(L_{\mathbb Q})\cap\mathbb Q^{2n}_{\geq 0}.
\]
To prove this it is sufficient to show that,
given $I_{\mathscr B} := \langle x^{a}-x^{b}\, :\, 
(a,b)\in \mathscr B\rangle$, the equality
$I_{\mathbb P} = I_{\mathscr B}$ holds.
If $(a,b)\in\mathscr B$ then $m := a-b\in L$ 
and both $a$ and $b$ have non-negative entries,
so that $x^{a}-x^{b}\in I_{\mathbb P}$. This proves the 
``$\supseteq$'' inclusion. On the other hand, 
if $x^p-x^q\in I_{\mathbb P}$ then $(p,q)\in\sigma_L$,
so that $(p,q) = \sum_{(a,b)\in \mathscr B}
c_{(a,b)}(a,b)$, where each $c_{(a,b)}$ is a 
non-negative integer. Thus in the quotient
ring $k[x_1,\dots,x_n]/I_{\mathscr B}$ we
have the following
\[
 \bar x^p 
 = \prod_{(a,b)\in \mathscr B}\bar x^{ac_{(a,b)}}
 = \prod_{(a,b)\in \mathscr B}\bar x^{bc_{(a,b)}}
 = \bar x^q,
\]
which implies $x^p-x^q\in I_{\mathscr B}$,
proving the ``$\subseteq$'' inclusion.
\end{remark}

\begin{theorem}
\label{thm:ts}
Given a projective toric surface $\mathbb P := \mathbb P_\Sigma$,
the following inclusions hold
\[
 {\rm Neg}(\mathbb P)\subseteq 
 {\rm wd}(\mathbb P) \subseteq 
{\rm Lib} (\mathbb P).
\]
\end{theorem}
\begin{proof}
We prove the first inclusion. Let us fix a lattice direction $v\in {\rm Neg}(\mathbb P)$, i.e. a $v\in N$ such that the class of the curve $\tilde{C}_v$ has negative self-intersection on the blowing up ${\rm Bl}_e\mathbb P$.
The class of $\tilde{C}_v$ spans an extremal ray of the effective cone ${\rm Eff}({\rm Bl}_e\mathbb P)$. 
Let us consider the facet of the nef cone 
${\rm Nef}({\rm Bl}_e\mathbb P)$
orthogonal to $\tilde{C}_v$ and let us fix a 
divisor $D$ whose class is in the relative 
interior of this facet. 
By construction $D\cdot \tilde C_v = 0$ and $D\cdot \tilde C_w > 0$ for any $w\neq v$. 
Given an ample divisor $A$ on ${\rm Bl}_e\mathbb P$
and a sufficiently small $\varepsilon 
\in {\mathbb Q}_{>0}$, the divisor
$D + \varepsilon A$ is ample and 
$\min\{(D + \varepsilon A)\cdot \tilde C_u\, :\, u\in N\}$ is attained at $\tilde C_v$.
Thus the pushforward of $D + \varepsilon A$ 
on $X$ is an ample divisor $H$ such that
${\rm lw}(\Delta_H) = {\rm lw}_v(\Delta_H)$. 

We prove the second inclusion.
Let $\{f_1,\dots,f_s\}$ be a minimal
binomial basis of the lattice ideal and
let $\pm v_1,\dots,\pm v_s\in N$ be the
corresponding directions.
Let $v\in {\rm wd}(\mathbb P)$ and let 
$x^a-x^b$ be a binomial for $C_v$.
Then $x^a-x^b = \sum_{i=1}^s
g_if_i$, where the $g_i$ are 
homogeneous polynomials. Thus $C_v$ is 
linearly equivalent to $C_{v_i}+E$, where
$E$ is the effective divisor defined by $g_i$.
Let $D$ be an ample divisor such that 
$v$ is a width direction of the polytope
$\Delta_D$. Then ${\rm lw}_v(\Delta_D)
= D\cdot C_v = D\cdot C_{v_i} + D\cdot E
= {\rm lw}_{v_i}(\Delta_D) + D\cdot E$
implies $D\cdot E = 0$, so that $E = 0$.
Thus $v = v_i\in {\rm Lib} (\mathbb P)$.


\end{proof}

\begin{remark}
    By~\cite{PS}*{Thm.~3.7}, if the 
    lattice ideal $I_{\mathbb P}$ is
    not a complete intersection, then
    we have the equality
    ${\rm Lib}(\mathbb P) = 
    {\rm Hbs}(\mathbb P)^{\pm}$. In particular for
    every direction $v\in {\rm wd}(\mathbb
    P)$ we have that both $v$ and $-v$
    belong to the Hilbert basis of a
    cone of $\Sigma$.
\end{remark}
Both inclusions in Theorem~\ref{thm:ts}
can be strict as shown in the following
example.

\begin{example}
Let $\mathbb P := \mathbb P_\Sigma$ be the toric surface whose 
fan $\Sigma\subseteq\mathbb Q^2$ is given in 
the first of the following pictures.

\begin{center}
\begin{tikzpicture}[scale=.5]
\tkzDefPoint(0,0){O}
\tkzDefPoint(0,1){A}
\tkzDefPoint(-1,2){B}
\tkzDefPoint(-1,-1){C}
\tkzDefPoint(3,-2){D}

\fill[gray!30] (O) -- (A) -- (B) -- cycle; 
\fill[gray!30] (O) -- (B) -- (C) -- cycle; 
\fill[gray!30] (O) -- (C) -- (D) -- cycle; 
\fill[gray!30] (O) -- (D) -- (A) -- cycle; 

\tkzDrawSegments(O,A O,B O,C O,D)
\tkzDrawPoints[size=2](O,A,B,C,D)
\draw[help lines,densely dotted] (-1.2,-2.2) grid (3.2,2.2);

\begin{scope}[xshift=6cm]
\tkzDefPoint(0,0){O}
\tkzDefPoint(0,1){A}
\tkzDefPoint(-1,2){B}
\tkzDefPoint(-1,-1){C}
\tkzDefPoint(3,-2){D}

\tkzDefPoint(0,1){P2}

\fill[gray!30] (O) -- (A) -- (B) -- cycle; 
\fill[gray!30] (O) -- (B) -- (C) -- cycle; 
\fill[gray!30] (O) -- (C) -- (D) -- cycle; 
\fill[gray!30] (O) -- (D) -- (A) -- cycle; 

\tkzDrawSegments(O,A O,B O,C O,D)
\tkzDrawSegments[color=red](O,P2)
\tkzDrawPoints[size=2](O,A,B,C,D)
\tkzDrawPoints[size=2, color=red](P2)
\draw[help lines,densely dotted] (-1.2,-2.2) grid (3.2,2.2);
\end{scope}

\begin{scope}[xshift=12cm]
\tkzDefPoint(0,0){O}
\tkzDefPoint(0,1){A}
\tkzDefPoint(-1,2){B}
\tkzDefPoint(-1,-1){C}
\tkzDefPoint(3,-2){D}

\tkzDefPoint(1,-1){P1}
\tkzDefPoint(0,1){P2}

\fill[gray!30] (O) -- (A) -- (B) -- cycle; 
\fill[gray!30] (O) -- (B) -- (C) -- cycle; 
\fill[gray!30] (O) -- (C) -- (D) -- cycle; 
\fill[gray!30] (O) -- (D) -- (A) -- cycle; 

\tkzDrawSegments(O,A O,B O,C O,D)
\tkzDrawSegments[color=green](O,P1 O,P2)
\tkzDrawPoints[size=2](O,A,B,C,D)
\tkzDrawPoints[size=2, color=green](P1,P2)
\draw[help lines,densely dotted] (-1.2,-2.2) grid (3.2,2.2);
\end{scope}

\begin{scope}[xshift=18cm]
\tkzDefPoint(0,0){O}
\tkzDefPoint(0,1){A}
\tkzDefPoint(-1,2){B}
\tkzDefPoint(-1,-1){C}
\tkzDefPoint(3,-2){D}

\tkzDefPoint(1,-1){P1}
\tkzDefPoint(0,1){P2}
\tkzDefPoint(1,0){P3}

\fill[gray!30] (O) -- (A) -- (B) -- cycle; 
\fill[gray!30] (O) -- (B) -- (C) -- cycle; 
\fill[gray!30] (O) -- (C) -- (D) -- cycle; 
\fill[gray!30] (O) -- (D) -- (A) -- cycle; 

\tkzDrawSegments(O,A O,B O,C O,D)
\tkzDrawSegments[color=blue](O,P1 O,P2 O,P3)
\tkzDrawPoints[size=2](O,A,B,C,D)
\tkzDrawPoints[size=2, color=blue](P1,P2,P3)
\draw[help lines,densely dotted] (-1.2,-2.2) grid (3.2,2.2);
\end{scope}
\end{tikzpicture}
\end{center}
The lattice ideal $I_\Sigma$ is
generated by the three binomials
$x_2x_3-x_4^3, x_1x_2x_4-x_3^2, x_1x_2^2-x_3x_4^2$
which corresponds to the three directions in
${\rm Lib} (\mathbb P) = \{\pm(0,1),\allowbreak
\pm(1,-1),\allowbreak \pm(1,0)\}$.
The grading matrix for the Cox ring of ${\mathbb P}$ is
\[
 \begin{bmatrix}
  0&5&4&3\\
  1&3&3&2
 \end{bmatrix}.
\]
The nef cone is ${\rm Nef}({\mathbb P}) = 
{\rm Cone}([3,2],[4,3])$ and $\Sigma_{\rm Nef}({\mathbb P})$
has two maximal cones:
$\sigma_1 := {\rm Cone}([3,2],[15,11])$ and 
$\sigma_2 := {\rm Cone}([15,11],[4,3])$.
The class $[15,11]\in \sigma_1\cap\sigma_2$ 
corresponds to the torus-invariant Weil divisor 
$D := 3D_3+D_4$. The divisor $3D$ is 
Cartier and its Riemann-Roch polytope
\[
 \Delta_{3D}
 =
 {\rm Polytope}((0,0),(3,6),(6,3),(-1,0))
\]
has width $6$, attained along the directions 
$\pm(0,1),\pm(1,-1)$. Polytopes corresponding 
to Cartier classes in the interior of 
$\sigma_1$ or of $\sigma_2$ admit just
one of these two width directions.
It follows that ${\rm wd}(\mathbb P) = \{\pm(0,1),\pm(1,-1)\}$.
Finally the curve $\tilde C_v$, with $v = (0,1)$,
is the only one with negative 
self-intersection ($\tilde C_v^2 = -2/5$),
so that ${\rm Neg}(\mathbb P) = \{\pm(0,1)\}$.
\end{example}

\section{Lattice ideals}
\label{sec:li}
As before, let $\mathbb P := \mathbb P_\Sigma$ be a projective 
toric variety with fan $\Sigma\subseteq
\mathbb Q^n$.
Denote by $R := \mathbb K[x_1,\dots,x_r]$
the Cox ring of $\mathbb P$. The Cox ring 
$\mathcal R({\rm Bl}_e\mathbb P)$ of 
${\rm Bl}_e\mathbb P$ is isomorphic to
a subalgebra of $R[t,t^{-1}]$, where the 
extra variable $t$ represents the exceptional 
divisor (see~\cite{hkl}*{Prop. 5.2}).
\begin{definition}
\label{def:mul}
We say that $\mathcal R({\rm Bl}_e\mathbb P)
\subseteq R[t,t^{-1}]$ is 
generated in multiplicity $m$ if it admits a set of 
homogeneous generators of the form $ft^{-i}$ with 
$0\leq i\leq m$ and $f\in R$.
\end{definition}
In this section we study the problem
of determining when the Cox ring of 
the blowing-up ${\rm Bl}_e\mathbb P$
is generated in multiplicity one,
in terms of the lattice ideal 
$I_{\mathbb P}$. We begin with
a result holding for projective
toric varieties and then
in the rest of the section 
we will restrict to the case of 
toric surfaces.

\begin{proposition}
\label{lem:ci}
Let $\mathbb P$ be an $n$-dimensional 
projective toric variety, let 
$R := \mathcal R(\mathbb P)$ and let
$I_{\mathbb P} = 
\langle f_1,\dots,f_n\rangle\subseteq R$
be a complete intersection. Then 
\[
 \mathcal R({\rm Bl}_e\mathbb P)
 \simeq
 R[s_1,\dots,s_n,t]/
 \langle f_i-ts_i\, :\, 1\leq i\leq n\rangle.
\]
In particular the Cox ring  
is generated in multiplicity $1$.
\end{proposition}
\begin{proof}
Let $S := R[s_1,\dots,s_n,t]$ and let 
$J := \langle f_i-ts_i\, :\, 1\leq i\leq
n\rangle\subseteq S$. We have
\[
 \dim (J+\langle t\rangle) 
 = \dim S/(J+\langle t\rangle) 
 = \dim R[s_1,\dots,s_n]/I_{\mathbb P} 
 = \dim R.
\]
On the other hand, for any generator
$x_i\in R$ we have 
\[
 \dim (J+\langle t,x_i\rangle) 
 = \dim S/(J+\langle t,x_i\rangle) 
 = \dim R[s_1,\dots,s_n]/(I_{\mathbb P}
 +\langle x_i\rangle)
 < \dim R,
\]
where the last inequality is due to the fact
that the lattice ideal $I_{\mathbb P}$ is a graded-prime ideal. 
Then, by~\cite{hkl}*{Prop.~5.1}, the Cox ring
of ${\rm Bl}_e\mathbb P$ is isomorphic
to $R[s_1,\dots,s_n,t]/J^{\rm sat}$, where
$J^{\rm sat}$ is the saturation 
$J\colon\langle t\rangle^\infty$.
To conclude we need to show that 
$J = J^{\rm sat}$. Let us suppose by
contradiction that there exists a 
$g\in J^{\rm sat}
\setminus J$. If we denote by $\tilde g$
the evaluation of $g$ at $t=0$, we have
\[
 \dim S/(J^{\rm sat}+\langle t\rangle) 
 \leq
 \dim S/(J+\langle g,t\rangle) 
 =
 \dim
 R[s_1,\dots,s_n]/(I_{\mathbb P}+\langle \tilde g\rangle)
 = \dim R - 1.
\]
Since
$\dim R[s_1,\dots,s_n,t]/J^{\rm sat}
= \dim R[s_1,\dots,s_n,t^{\pm 1}]/J^{\rm sat}$
$= \dim R[t^{\pm 1}] = \dim R + 1$, we get a
contradiction.

\end{proof}

\subsection{Toric surfaces}
In what follows ``lattice ideal'' will
mean lattice ideal of a projective toric surface.
Let us recall the following result (see for instance~\cite{Kr}*{Theorem~1}
and~\cite{PS}*{Proposition~4.1}) about lattice ideals in $3$ variables.
\begin{lemma}
 \label{lem:A}
  Any lattice ideal $I_{\mathbb P}\subseteq
  k[x_1,x_2,x_3]$ is Cohen-Macaulay. Moreover 
  $I_{\mathbb P}$ is not a
  complete intersection if and only if 
  (modulo reordering
  the indexes) it is generated by the $2\times 2$
  minors of a matrix
  \begin{equation}
  \label{eq:A}
  A =
  \begin{pmatrix}
  x_1^{\alpha_1} & x_2^{\alpha_2} &
  x_3^{\alpha_3}\\
  x_3^{\beta_3} & x_1^{\beta_1} & x_2^{\beta_2}
  \end{pmatrix},
  \end{equation}
  with $\alpha_i,\beta_i > 0$ for any 
  $i=1,2,3$.
\end{lemma}
When there are only $3$
variables, or equivalently the fan $\Sigma$
has $3$ rays, we have the following 
characterization of
toric surfaces $\mathbb P
= \mathbb P_{\Sigma}$ such that the Cox ring
of ${\rm Bl}_e\mathbb P$ is generated in
multiplicity $1$
(see also~\cite{hkl}*{Thm.~1.1}).
\begin{proposition}
    \label{cor:ci}
      Let $L$ be a lattice in $\mathbb{Z}^3$ associated to the rays of a complete
      fan $\Sigma$, and let 
      $I_{\mathbb P}\subseteq k[x_1,x_2,x_3]$ be the corresponding lattice ideal. The following are equivalent.
      \begin{itemize}
      \item[(i)] The Cox ring of ${\rm Bl}_e\mathbb P$ is generated in multiplicity $1$.
      \item[(ii)] $I_{\mathbb P} \not \subseteq \langle x_1,x_2,x_3\rangle^2$.
      \item[(iii)] $I_{\mathbb P}$ is a complete intersection.
      \end{itemize}
\end{proposition}
\begin{proof}
The implication $(ii) \Rightarrow (iii)$ follows from Lemma~\ref{lem:A}.
We prove $(iii) \Rightarrow (ii)$. 
Let us suppose
that $I_{\mathbb P} \subseteq
\langle x_1,x_2,x_3\rangle^2$. 
Since $I_{\mathbb P}$ is a lattice ideal, 
it does not contain linear components of codimension $2$, and hence, up to reordering the variables, $I_{\mathbb P}$ is generated by
$f_1 = x_1^{a_1} - x_2^{a_2}x_3^{a_3}$
and $f_2 = x_2^{b_2} - x_3^{b_3}$, where
$a_1,b_2,b_3 > 1$. This implies that 
the lattice $L$ is generated by the rows 
of the following matrix
\[
\begin{pmatrix}
        a_1 & -a_2 & -a_3\\
        0 & b_2 & -b_3
\end{pmatrix},
\]
whose columns are the rays of the fan
$\Sigma$. Since the first
column is not primitive we get a contradiction.

The implication $(iii) \Rightarrow (i)$
is given by Lemma~\ref{lem:ci}, so that we conclude 
by showing that $(i) \Rightarrow (iii)$. If the 
Cox ring of ${\rm Bl}_e\mathbb P$ is generated in multiplicity $1$, by 
Corollary~\ref{cor:4C} there
there exist two one-parameter subgroups 
$C_v$ and $C_w$ such that $\tilde C_v\cdot \tilde C_w = 0$.
This implies that $C_v\cdot C_w = 1$, so that
$C_v$ and $C_w$ meet only at $e$, transversely.
Therefore $\det(v,w) = 1$ and we can assume 
without loss of generality that 
$v = (1,0)$ and $w  = (0,1)$. Let us denote 
by $\rho_1, \rho_2$ and $\rho_3$ the primitive
vectors generating the $1$-dimensional
rays of $\Sigma$. Since no
invariant point of $\mathbb P$ is contained in
the intersection $C_v\cap C_w$, one of the
two directions, say $v$, must lie on a 
$1$-dimensional ray of $\Sigma$, so that 
we can assume that $\rho_1 = v$.
The remaining two rays of $\Sigma$ must lie on the
second and third quadrant respectively (one of them
can lie on the vertical axis). 
By~\cite{PS}*{Lemma~3.1} we conclude that the lattice ideal 
$I_{\mathbb P}$ is a complete intersection.
\end{proof}
When the number of variables is bigger, we
have the following weaker result.
\begin{theorem}
\label{thm:gen}
If the Cox ring of ${\rm Bl}_e\mathbb P$
is generated in multiplicity $1$, then 
for any three indexes $i,j,k\in\{1,\dots,n\}$, $I_{\mathbb{P}}\not\subseteq \langle x_i,x_j,x_k
\rangle^2$.
\end{theorem}
\begin{proof}
    Let us suppose by contradiction that
    there are three indexes, say $1,2,3$, 
    such that $I_{\mathbb P}\subseteq 
    \langle x_1,x_2,x_3\rangle^2$, 
    and let us denote 
    by $\rho_1,\dots,\rho_n$ the primitive
    vectors generating the $1$-dimensional
    rays of $\Sigma$.   
    Let us fix a basis $B
    =\{f_1,\dots,f_r\}$
    for the lattice ideal $I_\mathbb{P}$, and let
    us consider the ideal 
    \[
    I' := \langle
    f_i'\, : \, 1\leq i\leq r
    \rangle\subseteq\mathbb{C}[x_1,x_2,x_3],
    \]
    where we set $f_i' := 
    f(x_1,x_2,x_3,1,\dots,1)$.
    Let $\Sigma'$ be the fan whose rays are
    generated by $\rho_1,\rho_2$ and
    $\rho_3$, let $L'\subseteq \mathbb{Z}^3$ be the
    corresponding lattice and let us denote
    by $\mathbb P'$ the toric surface 
    $\mathbb P_{\Sigma'}$. We claim that 
    $\Sigma'$ is complete. 
    First of all observe that
    the map $({\mathbb C}^*)^3 \to 
    ({\mathbb C}^*)^2$, given by the matrix
    of $L'$, factorizes through the map
    $({\mathbb C}^*)^n \to 
    ({\mathbb C}^*)^2$, given by $L$. 
    This implies that $I_{\mathbb P'}$ equals
    the saturation of $I'$ with respect to
    $x_1x_2x_3$. Since $I'
    \subseteq \langle x_1,x_2,x_3 \rangle^2$,
    the monomials of each binomial in the 
    saturation of $I'$ is divisible by 
    either $x_1,x_2$ or $x_3$. This gives
    the claim.
    We can then apply Proposition~\ref{cor:ci}
    to $I_{\mathbb P}'$ and deduce that the Cox 
    ring of ${\rm Bl}_e\mathbb P'$ is not 
    generated in multiplicity $1$.
    We conclude observing that the contraction
    of the rays $\{\rho_i\, : \, i = 
    4,\dots,n\}$ induces a surjective morphism
    $\pi\, \colon\, {\rm Bl}_e\mathbb P \to 
    {\rm Bl}_e\mathbb P'$.
\end{proof}
Indeed, in all the examples we checked 
with the help of Magma~\cite{mag} we
have found that also the other implication
holds, so that we formulate the following.
\begin{conjecture}
\label{con:m1}
Given a be a projective toric surface 
$\mathbb P$ the following 
are equivalent:
\begin{enumerate}
\item the Cox ring of ${\rm Bl}_e\mathbb P$
is generated in multiplicity $1$;
\item for any three indexes $i,j,k$,
$I_{\mathbb P}\not\subseteq \langle x_i,x_j,x_k\rangle^2$.
\end{enumerate}
\end{conjecture}
We have already seen that the above conjecture holds if there are only $3$ variables
(Proposition~\ref{cor:ci}) and we are going 
to show that it also holds if the ideal $I_{\mathbb P}$
is Cohen-Macaulay, i.e. it has at most $3$ minimal
generators.
\begin{theorem}
\label{thm:CM}
Conjecture~\ref{con:m1} holds true if  
$I_{\mathbb P}$ is Cohen-Macaulay.
\end{theorem}
\begin{proof}
By Theorem~\ref{thm:gen} we only need to prove the implication 
$(ii) \Rightarrow (i)$ of Conjecture~\ref{con:m1}.
If $I_{\mathbb P}$ has $2$ generators, 
then it is a complete 
intersection, so that the Cox ring of ${\rm Bl}_e
\mathbb P$ is generated in multiplicity $1$ by
Proposition~\ref{lem:ci}. Let us suppose
that there exist three
indexes, say $1,2,3$, such that $I_{\mathbb P}
\subseteq \langle x_1,x_2,x_3\rangle^2$.
Reasoning as in the proof of 
Proposition~\ref{cor:ci} we would have
that the generator of a ray of $\Sigma$ is
not primitive, a contradiction.

Let us suppose now that 
$I_{\mathbb P}$ has $3$ minimal 
generators,
so that, by Hilbert--Burch Theorem~\cite{E}*{Thm.~3.2}, 
$I_{\mathbb P} = \langle f_1,f_2,f_3\rangle$, where
$f_i$ are the maximal minors of a $2\times 3$ 
matrix $A = (a_{i,j})$, whose entries are monomials in $x_1,\dots,x_n$.
In this case we have that 
$J^{{\rm sat}} + \langle t\rangle \supseteq 
\langle f_1,f_2,f_3,h_1,h_2\rangle$,
where $h_1 = a_{1,1} s_1 + a_{1,2}s_2 + a_{1,3}s_3$ and
$h_2 = a_{2,1} s_1 + a_{2,2}s_2 + a_{2,3}s_3$.
Let us fix a variable $x_i$, with $1\leq i\leq n$.
Since $x_i$ appears in at least one of $f_1,f_2,f_3$, 
we have that the vanishing of $x_i$ implies the vanishing of at least another 
variable $x_j$. Let us distinguish two cases.
\begin{itemize}
\item[i)] The vanishing of $x_i$ and $x_j$ forces a third variable $x_k$ to
be zero. By the assumption $I_{\mathbb P}
\not \subseteq \langle x_i,x_j,x_k\rangle^2$ we deduce that there exists an 
entry of $A$ that does not vanish
for $x_i=x_j=x_k=0$. In particular, at least one of $h_1,h_2$ does not vanish
identically, so that $\dim(J^{{\rm sat}} + \langle t,x_i\rangle) \leq n-1$. 

\item[ii)] The vanishing $x_i = x_j = 0$ does not imply the vanishing of a third variable.
In this case there is at least one of the binomials, say $f_1$, whose monomials do not 
contain $x_i$ and $x_j$, so that it survives when we set $x_i=x_j=0$. Moreover,
in each row of $A$ there is at least one entry that does not vanish, which implies that
also $h_1$ and $h_2$ do not vanish identically.
This implies again that $\dim(J^{{\rm sat}} + \langle t,x_i\rangle) \leq n-1$
\end{itemize}
Since the above reasoning holds for any index $1\leq i\leq n$, we conclude that 
$\dim(J^{{\rm sat}} + \langle t,\Pi_1^nx_i\rangle) \leq n-1$, so that the Cox ring is generated in 
multiplicity $1$.

\end{proof}


\section{Minimal toric surfaces}
\label{sec:min}
Let us suppose that $\mathbb P$ is a 
minimal toric surface. This is equivalent 
to say that the fan $\Sigma$ has either
$3$ rays, or $4$ rays generated by 
primitive vectors
$\pm u,\, \pm v$. 
The former case has 
been studied in Proposition~\ref{cor:ci},
so that from now on we focus on the latter case.
Modulo ${\rm GL}(2,{\mathbb Z})$ we can always suppose that 
the $1$-dimensional cones of the normal fan $\Sigma$ are generated by the columns 
of the matrix
\begin{equation}
    \label{eq:mat}
\begin{pmatrix}
1 & p & -1 & -p\\
0 & q & 0 & -q
\end{pmatrix}
\end{equation}
where $0 < p \leq q/2$ and $\gcd(p,q) = 1$.
The intersection matrix with respect to the basis $(H_1,H_2,E)$ of the Picard
group of the blow-up $\rm{Bl}_e\mathbb P$ 
is
\begin{equation}
\label{eq:M}
M =
 \begin{pmatrix}
 0 & \frac1{q} & 0\\
 \frac1{q} & 0 & 0\\
 0 & 0 & -1
 \end{pmatrix}.
\end{equation}
Let us fix a direction $v=(a,b)\in N$, with 
$b\geq 0$. The class of the
strict transform $\tilde C_v$ of the corresponding
$1$-parameter subgroup is 
\begin{equation}
\label{eq:class}
 bH_1 + |aq-bp| H_2 -E.
 \end{equation}
In what follows we are 
going to consider only a suitable 
set of $1$-parameter subgroups.
In order to define them we first 
need the following construction.
\begin{construction}
\label{con:conv}
Let $[c_1,c_2,\dots, c_n]$ be the Euclidean 
continued fraction of $p/q$ and 
let $a_i/b_i$, for $1\leq i\leq n$ be the convergents of $p/q$.  
It is possible to describe
recursively $a_i$ and $b_i$ by means of the following (see for instance~\cite{Pe})
\begin{equation}
\label{eq:ab}
\begin{array}{ll}
\left\{
\begin{array}{rcl}
 a_{-1} & = & 0\\ 
 a_0 & = & 1\\
 a_{i} & = & c_{i}a_{i-1}+a_{i-2}
 \end{array}
 \right.
&
 \left\{
 \begin{array}{rcll}
 b_{-1} & = & 1\\ 
 b_0 & = & 0\\
 b_{i} & = & c_{i}b_{i-1}+b_{i-2}, & 1\leq i\leq n.
\end{array}
\right.
\end{array}
\end{equation}
If we set $v_i := (a_i,b_i)\in N$, 
for $0\leq i\leq n$, we have that
$v_i$ lies in the cone $\langle (1,0),(p,q)\rangle$ if $i$ is even,
while it lies in $\langle (p,q),(-1,0)\rangle$ if $i$ is odd. Let us denote simply by $C_i$ 
the class of $\tilde C_{v_i}$ in 
${\rm Bl}_e\mathbb P$. From~\eqref{eq:class}
we have that
\[
 C_i = b_iH_1 + \beta_iH_2 - E,
 \quad
 \text{where}
 \quad
 \beta_i := (-1)^i(a_iq-b_ip).
\]

\end{construction}

\begin{lemma}
\label{lem:rec}
With the notation above, the following hold:
\begin{enumerate}
\item $\beta_0 = q$,\ $\beta_1 = p$ and $
\beta_{i} = \beta_{i-2} - c_{i}\beta_{i-1}$ for $2\leq i\leq n$.
\item $\beta_ib_{i+1} + b_i\beta_{i+1} = q$, for any $0\leq i \leq n-1$.
\item $\beta_{i-1}b_{i+1} - \beta_{i+1}b_{i-1} = c_{i+1}q$, for any $1\leq i \leq n-1$.
\item $\beta_{i}b_{i} + \beta_{i+1}b_{i+1} < q$, for any $0\leq i \leq n-1$.
\end{enumerate}
\end{lemma}
\begin{proof}
We only prove (iv), since (i), (ii) and (iii)
easily follow by induction. 
By (ii) and (i) we can write
\[
 q - \beta_ib_i - \beta_{i+1}b_{i+1} = \beta_ib_{i+1} + b_i\beta_{i+1} - 
 \beta_ib_i - \beta_{i+1}b_{i+1}
 = (\beta_{i+1}-\beta_{i})(b_{i}-b_{i+1}).
\]
Since we are supposing $2p < q$ 
we have that $c_1 = 0$ and $c_2 > 1$,
so that $\{b_0,\dots,b_n\}$ is strictly
increasing. In a similar way, by (i) one can 
deduce that $\{\beta_0,\dots,\beta_n\}$ 
is strictly decreasing.
The statement follows.
\end{proof}
\begin{remark}
\label{rem:ort}
Using the intersection matrix~\eqref{eq:M} 
and Lemma~\ref{lem:rec}(ii) we have that for any $0\leq i \leq n-1$,
\[
C_i\cdot C_{i+1} = \frac{\beta_{i}b_{i+1}+\beta_{i+1}b_{i}}q - 1 = 0.
\]
\end{remark}

\subsection{Effective cones}
We are now going to prove that the 
effective cone 
${\rm Eff}({\rm Bl}_e\mathbb P)$
is generated in multiplicity $1$.
In order to do this, we are going to use
Proposition~\ref{lem:eff} with the tuple 
consisting
of $E$ and all the classes $C_i$ having
negative self intersection. 
\begin{theorem}
 \label{thm:eff}
With the notation above,
  \[
 {\rm Eff}({\rm Bl}_e\mathbb P) = {\rm cone}(\{E\} \cup \{C_i \mid 0\leq i \leq n
\ {\rm and \ } 2\beta_ib_i < q\}),
 \]
 in particular it is generated in multiplicity $1$.
\end{theorem}
\begin{proof}
Let us consider the tuple 
$\mathcal S$ consisting of $E$ and
all the $C_i$ with $0\leq i \leq n$,
satisfying $2\beta_ib_i < q$. 
Since $C_i^2 = 2\beta_ib_i/q -1$, 
every $C_i \in \mathcal S$ is the class of a negative curve. Therefore,
by Proposition~\ref{lem:eff}, in order to prove our result it is enough to 
show that 
the intersection matrix is negative semidefinite
on every cone generated by a pair of consecutive 
classes in $\mathcal S$. 

\noindent First of all, $C_0 = qH_2-E$ belongs to $\mathcal S$.
The matrix of the intersection form on ${\rm Cone}( E, C_0 )$
is $ \begin{pmatrix}
 -1 & 1\\
 1 & -1
 \end{pmatrix}$, 
 which is negative semidefinite, and the same holds 
 for ${\rm Cone}(C_{n},E)$.


Let us now fix any index $0\leq j\leq n$, for which 
$2\beta_jb_j < q$, so that $C_j \in\mathcal S$.
If we have that also 
$2\beta_{j+1}b_{j+1} < q$, then the 
next class in $\mathcal S$ is $C_{j+1}$. By Remark~\ref{rem:ort},
the matrix of the intersection form on ${\rm Cone}( C_j,C_{j+1})$ is
diagonal and hence it is negative definite. 
If otherwise $2\beta_{j+1}b_{j+1} \geq q$, by Lemma~\ref{lem:rec}(iv) we must have
$2\beta_{j+2}b_{j+2} < q$, i.e. the next class in $\mathcal S$ is $C_{j+2}$. 
The matrix of the intersection form on ${\rm Cone}( C_j,C_{j+2})$ is
\[
\begin{pmatrix}
\frac{2\beta_{j}b_{j}}q - 1 & \frac{\beta_{j+2}b_{j}+\beta_{j}b_{j+2}}q - 1\\
\frac{\beta_{j+2}b_{j}+\beta_{j}b_{j+2}}q  - 1 & \frac{2\beta_{j+2}b_{j+2}}q - 1
\end{pmatrix}.
\]
Its determinant can be written as
\[
 2\frac{(\beta_{j}-\beta_{j+2})
 (b_{j+2}-b_{j})}q
 -\frac{(b_{j+2}\beta_j - \beta_{j+2}b_j)^2}
 {q^2}
 = \frac{c_{j+2}^2}q(2\beta_{j+1}b_{j+1}-q) 
\]
where the equality follows from 
Lemma~\ref{lem:rec}(iii) 
and~\eqref{eq:ab}.
Since we are supposing that $2\beta_{j+1}b_{j+1} \geq q$, the above determinant 
is non-negative, which implies that the intersection form is negative semi-definite.

We conclude that the tuple 
$\mathcal S$ satisfy the hypotheses of 
Proposition~\ref{lem:eff} so that its classes
are the extremal rays of ${\rm Eff}(
{\rm Bl}_e\mathbb P)$.
\end{proof}

A consequence of the above theorem 
is the following characterisation of
the directions in ${\rm Neg}(\mathbb P)$
and ${\rm wd}(\mathbb P)$ in the cases
we are considering.

\begin{corollary}
With the notation above we have that:
\begin{enumerate}
\item ${\rm Neg}(\mathbb P) = 
    \{v_i\, : \, 0\leq i \leq n \ {\rm and }\ 
    2\beta_ib_i < q\}
=
\{v_i \, : \, C_i^2 < 0\}$;
\item ${\rm wd}(\mathbb P) = 
    \{v_i\, : \, 0\leq i \leq n \ {\rm and }\ 
    2\beta_ib_i \leq q\}
=
\{v_i \, : \, C_i^2 \leq 0\}$.
\end{enumerate}
\end{corollary}
\begin{proof}
 The assertion (i) is an immediate consequence of
 Theorem~\ref{thm:eff}. 
 Let us prove (ii).
 First of all we fix an index $0\leq i\leq n$
 such that $v_i\in{\rm Neg}(\mathbb P)$
 and we denote by $F_i$ the facet of ${\rm Nef}({\rm Bl}_e\mathbb P)$ orthogonal to the extremal 
 ray of ${\rm Eff}({\rm Bl}_e\mathbb P)$ 
 generated by $C_{i}$.
 The projection of $F_i$ gives a $2$-dimensional cone of the fan $\Sigma_{{\rm Nef}}(\mathbb P)$
 containing all the ample classes $D$ such that the only width direction of the
 corresponding Riemann-Roch polytope is $v_i$.
 
 In order to conclude the proof we need to show that the only new width directions
 that can occur in correspondence with a 
 an extremal ray $R$ of ${\rm Nef}({\rm Bl}_e\mathbb P)$, are the $v_i$ such that $C_i^2 = 0$,
 i.e. $2b_i\beta_i = q$. We have to distinguish two cases.
 
 \begin{enumerate}
 \item $R = F_i\cap F_{i+1}$. 
 This means that both $C_i$ and $C_{i+1}$ generate
 a ray of ${\rm Eff}({\rm Bl}_e\mathbb P)$, i.e.
 $v_i,\, v_{i+1}\in{\rm Neg}(\mathbb P)
 \subseteq {\rm wd}(\mathbb P)$. 
 Let $D$ be a divisor whose class generates
 the ray $R$. We claim that the Newton polytope
 of $D$ has no other width directions but 
 $v_i$ and $v_{i+1}$. Indeed, let us denote 
 by $\Delta_{m,n}$ the Newton polytope  with 
 vertices $(0,0),\, (mq,-mp),\, (mq,n-mp),\, (0,n)$. Computing the inner product of $v_j$ 
 with the vertices it is easy to see that 
 ${\rm lw}_{v_j}(\Delta_{m,n})
 = \beta_jm+b_jn$, for any $0\leq j\leq n$. 
 When we consider the ray $R = F_i\cap F_{i+1}$, we must have 
 \[
  {\rm lw}_{v_i}(\Delta_{m,n}) = {\rm lw}_{v_{i+1}}(\Delta_{m,n}) = \beta_ib_{i+1} - \beta_{i+1}b_i,
 \]
since $m = b_{i+1} - b_i$ and $n = \beta_i - \beta_{i+1}$. 
By Corollary~\ref{cor:3wd}, if $\Delta_{m,n}$
has a third width direction $v$,
it must be either 
$v = v_{i+1} - v_{i}$ or $v = v_{i+1} + v_i$. In the first case 
we have ${\rm lw}_v(\Delta_{m,n}) = 2\beta_i(b_{i+1}-b_i)$, which 
implies
 \[
 2\beta_i(b_{i+1}-b_i) = \beta_ib_{i+1}-\beta_{i+1}b_i
 \ \Rightarrow \ 2\beta_ib_i = \beta_ib_{i+1}+\beta_{i+1}b_i
 = q,
 \]
 where the last equality follows 
 from Lemma~\ref{lem:rec}(ii). 
 This is equivalent to $C_i^2 = 0$, contradicting the assumption $C_i^2 < 0$. 
 
 In a similar way, if $v = v_{i+1} + v_i$, we have ${\rm lw}_v(\Delta_{m,n}) =
 2b_{i+1}(\beta_i - \beta_{i+1})$ and 
 the equality ${\rm lw}_v(\Delta_{m,n})
 = {\rm lw}_{v_i}(\Delta_{m,n})$  
 implies $2\beta_{i+1} b_{i+1} = q$.
 This is equivalent to $C_{i+1}^2 = 0$,
 again a contradiction.

 \item $R = F_{i-1}\cap F_{i+1}$. In this case, ${\rm lw}_{v_{i+1}}(\Delta_{m,n})
 = {\rm lw}_{v_{i-1}}(\Delta_{m,n})$ implies $m/n = (b_{i+1} - b_{i-1})/(\beta_i - \beta_{i+1}) 
 = b_i/\beta_i$, so that we can set $m = b_i,\, n=\beta_i$ and
 \[ 
 {\rm lw}_{v_{i-1}}(\Delta_{m,n}) =
 {\rm lw}_{v_{i+1}}(\Delta_{m,n}) = \beta_{i+1}b_i + b_{i+1}\beta_i = q,
 \]
again by Lemma~\ref{lem:rec}(ii). 
As before we only have to check  
the directions $v_{i+1} + v_{i-1}$ and
$v_{i+1} - v_{i-1}$. 
We claim that the former can
not give a new width direction for $\Delta_{m,n}$. Indeed both $v_{i-1}$ 
and $v_{i+1}$ lie either in the 
cone having rays $(1,0),\, (p,q)$
or in the cone with rays 
$(-1,0),\, (p,q)$. 
In particular they belong to the
Hilbert basis of the same cone.
Since the sum $v$ of these two
vectors can not belong to the 
Hilbert basis, we conclude that
$v$ can not be a width direction.
Finally, if $v = v_{i+1} - v_{i-1}$,
we have
\[ 
v = (a_{i+1} - a_{i-1}, b_{i+1} - b_{i-1}) = c_{i+1}(a_i,b_i) = c_{i+1}v_i. 
\] 
Therefore it must be $c_{i+1} = 1$ and the width in this direction is 
${\rm lw}_{v_i} = \beta_im + b_in = 2\beta_ib_i$. We conclude that 
$v_i$ is a new width direction if and only if $2\beta_ib_i = q$, which means that $C_i^2 = 0$.
 
 \end{enumerate}

\end{proof}

\begin{remark}
 \label{rem:0}
 The above result implies that if 
 $C_i^2$ never vanishes, then 
 ${\rm wd}(\mathbb P) = {\rm Neg}(\mathbb P) = \{v_i \,:\, C_i^2 < 
 0\}$. We are now going to see that 
 the condition $C_j^2 = 0$ for some 
 $0\leq j \leq n$ turns out to be 
 quite strong. Indeed
 we know that it is equivalent to $2b_j
 \beta_j = q$, so that, by definition of
 $\beta_j$ we can write $pb_j = q a_j + (-1)^j\beta_j$. Therefore
\begin{equation}
 \label{eq:k}
  p  = \frac{q a_j + (-1)^j\beta_j}{b_j} = \frac{\beta_j(2a_jb_j+(-1)^j )}{b_j},
\end{equation}
 so that $b_j$ divides $\beta_j$, but since we are supposing $\gcd(p,q) = 1$,
 it must be $\beta_j = b_j$.
 In particular, since
 $\{b_i\}$ is strictly increasing and $\{\beta_i\}$ is strictly decreasing,
 there can be at most one index $j$ such that $b_j = \beta_j$, i.e. there is
 at most one curve $C_j$ such that 
 $C_j^2 = 0$. 
 Moreover, from $b_j = \beta_j$
 we deduce that $q = 2b_j^2$
 and from~\eqref{eq:k} that 
 $p = 2a_jb_j + (-1)^j$.
 
 On the other hand, let us fix
 two positive integers $r,s$ with
 $\gcd(r,s) = 1$ and $2s \leq r$, and 
 set $q = 2 r^2$ and $p= 2rs \pm 1$.
 Under these hypotheses, 
 by~\cite{Pe}*{Lemma} we have that 
 $s/r$ is a convergent for the continued fraction of $p/q$,
i.e. there exists an index $j$ such that 
$a_j = s$ and $b_j = r$.  
This implies that $\beta_j = r$, 
so that the 
curve $C_j = \beta_j H_1 + b_j H_2 -E$ satisfies $C_j^2 = 0$.

  
 We also remark that the matrix of the
 intersection form on any facet $F$ of ${\rm Eff}({\rm Bl}_e\mathbb P)$ is negative definite, 
 unless $F = {\rm Cone}( E,eH_i-E)$, 
 for $i = 1,2$, or $F = {\rm Cone}( C_{j-1},C_{j+1})$, where 
 $j$ is such that $C_j^2 = 0$. 
 Therefore the light cone $Q$ is tangent to
 ${\rm Eff}({\rm Bl}_e\mathbb P)$ either at the $2$ points $H_1,\, H_2$
 or at the $3$ points $H_1,\, H_2,\, C_j$, if $C_j^2 = 0$.
 
 \end{remark}

\begin{example}
 Let us fix $r = 11,\ s = 3$ 
 and consider the 
 toric surface $X$ with fan generated by
 the rays $(\pm 1,0),\, (\pm p,\pm q)$, where 
 $p = 2rs - 1 = 65$ and 
 $q = 2r^2 = 242$. 
 Let us denote by $c = [ 0, 3, 1, 2, 1, 1, 1, 1, 3 ]$ the continued fraction of $p/q$, so that 
 $b = [ 0, 1, 3, 4, 11, 15, 26, 41, 67, 242 ]$ and 
 $\beta = [ 242, 65, 47, 18, 11, 7, 4, 3, 1, 0 ]$ give rise
 to the classes $C_i = \beta_iH_1 + b_iH_2-E$, for $i = 0,\dots, 9$.
We have that for any $i\neq 2,4,7$, $2b_i\beta_i < q$ while $2b_2\beta_2 > q$, 
$2b_7\beta_7 > q$,
and $\beta_4 = b_4 = 11$, so that $2b_4\beta_4 = q$.
Therefore the light cone is tangent to ${\rm Eff}({\rm Bl}_e\mathbb P)$
at the points corresponding to $H_1,H_2$ and $C_4=11H_1+11H_2-E$
(see Figure~\ref{fig:eff}).

\begin{figure}[htbp]
\hspace{1mm}
\includegraphics[scale=0.26]{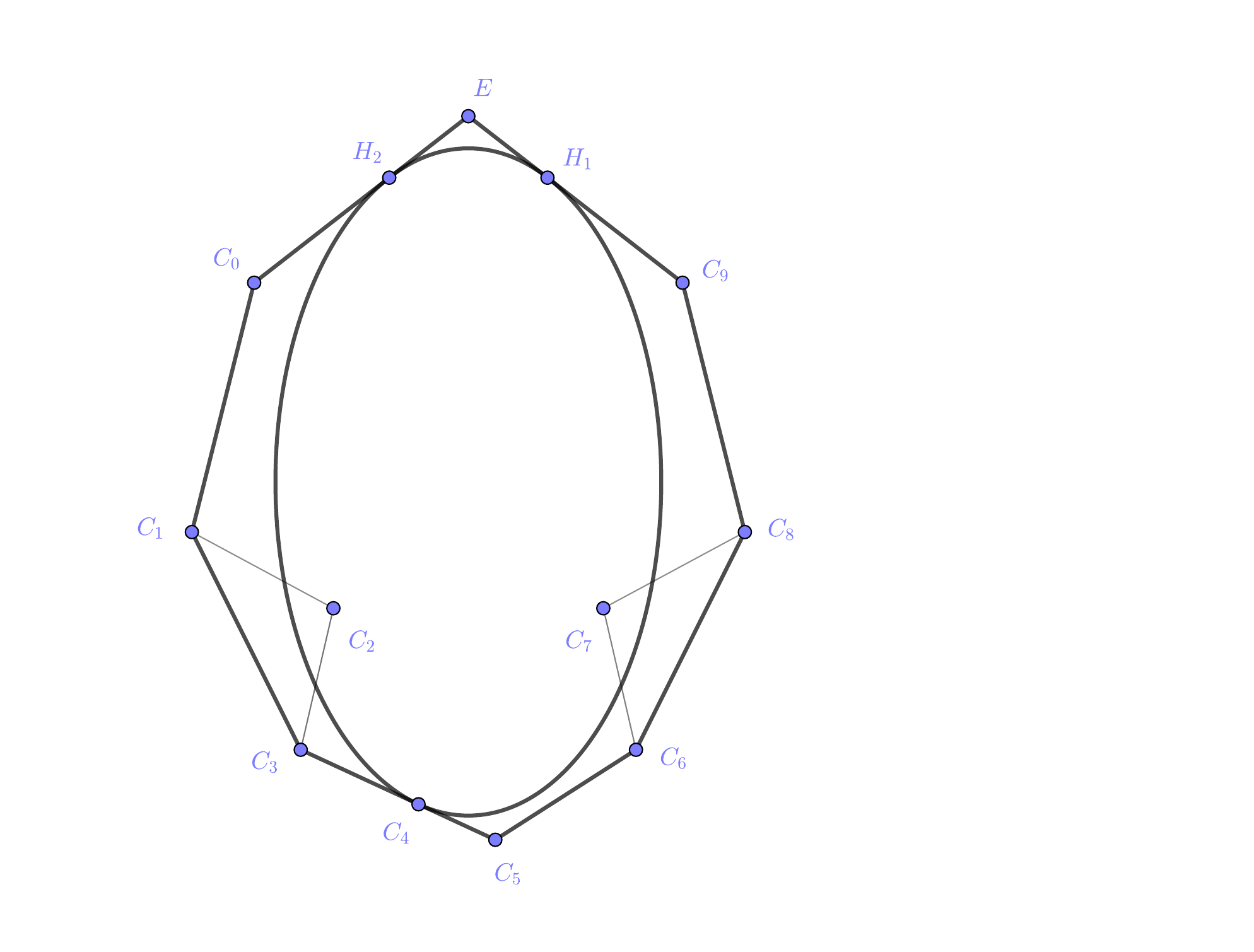}
\caption{${\rm Eff}({\rm Bl}_e\mathbb P)$} 
\label{fig:eff}
\end{figure}

\end{example}

\begin{example}
 \label{ex:fib}
Let us consider the case in which 
$(p,q) 
=
(\varphi_{n-1},
\varphi_{n+1})$, where we
denote by
$\varphi_n$ the $n$-th Fibonacci number\footnote{We
set $\varphi_0 = 0,\,
\varphi_1 
= 1$, and then $\varphi_{n+1}
= \varphi_{n-1}+\varphi_n$, 
for any $n\geq 1$.}, and $n \geq 5$. The continued
fraction of $p/q$ is 
$c = [0,2,1,\dots,1,2]$,
with $|c| = n-1$. 
Therefore $b_0 = 0$,\, 
$b_i = \varphi_{i+1}$ 
for any $1\leq i\leq n-2$,
and $b_{n-1} = \varphi_{n+1}$.
Analogously $\beta_0 = \varphi_{n+1},\,
\beta_i = \varphi_{n-i}$ for any $1\leq i\leq n-2$
and $\beta_{n-1} = 0$.
We claim that $2\beta_ib_i < q$ for any $0\leq i \leq 
n-1$. Indeed, 
if $i = 0$ or $n-1$ 
we have $2\beta_ib_i = 0$,
while if $1\leq i \leq n-2$
we can write 
\begin{align*}
 q - 2\beta_ib_i & = 
  \varphi_{n+1} - 
 2\varphi_{n-i}
 \varphi_{i+1} \\
 & = 
 \varphi_{i+1}\varphi_{n+1-i}
 +\varphi_i\varphi_{n-i}
 - 
 2\varphi_{n-i}
 \varphi_{i+1} \\
 & =
 \varphi_{i+1}\varphi_{n-1-i}
 - 
 \varphi_{i-1}
 \varphi_{n-i} \\
 & = 
 \varphi_{i}\varphi_{n-i-1} - \varphi_{i-1}\varphi_{n-i-2} 
 > 0,
 \end{align*}
where the second equality follows
from~\cite{V}*{Eqn.~(8)}.
Therefore, by Proposition~\ref{thm:eff}, all the classes 
$\beta_iH_1 + b_iH_2 - E$, for $i\in \{0,\dots,n-1\}$ are negative curves,
and, together with $E$, they generate the 
extremal rays of the effective cone. 

In particular this example shows that, given any positive integer $r \geq 3$, 
it is always possible to find a toric surface $\mathbb{P}$ associated to a parallelogram,
such that the effective cone of its blow-up in a general point has $r$ extremal rays. 

\end{example}

\subsection{Cox rings}
We are now going to see that the classes 
$C_i$ defined in Construction~\ref{con:conv} 
can also be used 
to prove Theorem~\ref{thm:main}, i.e.
that the Cox ring
${\rm Bl}_e\mathbb P$ is finitely generated.
\begin{proof}[Proof of Theorem~\ref{thm:main}]
Let us set
$\mathcal S := 
(E,H_1,H_2,C_0,\dots,C_{n})$
(we are considering here also the classes 
$C_i$ with non-negative self intersection).
We are going to use Proposition~\ref{lem:mah},
by showing that $\mathcal S$
is a pseudogenerating tuple for 
${\rm Bl}_e\mathbb P$ (see 
Definition~\ref{def:main}). The first
condition easily follows from
Theorem~\ref{thm:eff}. 
Let us prove the second one, i.e.
that for any class $D$ generating an extremal 
ray of ${\rm Nef}({\rm Bl}_e\mathbb P)$ 
we have that 
\begin{equation}
\label{eq:cond}
D \in \bigcap_{C\in \mathcal S\cap 
D^\perp}{\rm Cone}(\mathcal S\setminus C).
\end{equation}
By Theorem~\ref{thm:eff}, the facets of
${\rm Eff}({\rm Bl}_e\mathbb P)$
are ${\rm Cone}( E,C_0 ),\,{\rm Cone}( C_{n},E)$ 
and, for any $i$ such that 
$2\beta_ib_i < q$, we have either 
the facet ${\rm Cone}( C_i,C_{i+1} )$ (if also $2\beta_{i+1}b_{i+1} < q$) or 
${\rm Cone}( C_i,C_{i+2} )$ (if $2\beta_{i+1}b_{i+1} \geq q$).
Therefore we have to consider $3$ 
different cases.
\begin{center}
\begin{tikzpicture}[scale=0.4]
\begin{scope}
%
%

\tkzDefPoint(0,0){A}
\tkzDefPoint(2,1.5){B}
\tkzDefPoint(5,1.5){C}
\tkzDefPoint(7,0){D}
\tkzDefPoint(-.5,-2){E}
\tkzDefPoint(7.5,-2){F}
\tkzDefPoint(1,.75){G}
\tkzDefPoint(3.5,1.5){H}

\tkzDefPoint(3.5,-2){I}

\draw[line width=0mm,fill=gray!40!white] (-.5,-2) -- (0,0) -- (1,.75) -- (3.5,1.5) -- (7,0) -- (7.5,-2) -- (-.5,-2);

\tkzDrawPoints[size=2](A,B,C,D,G,H);
\tkzDrawSegments[dashed](E,A F,D G,H H,D)
\tkzDrawSegments[](A,B C,D)
\tkzDrawSegments[thick](B,C)

\tkzLabelPoint[left](A){\tiny $C_{n+2}$}
\tkzLabelPoint[left, above](H){\tiny $H_2$}
\tkzLabelPoint[right](C){\tiny $C_0$}
\tkzLabelPoint[left](G){\tiny $H_1$}
\tkzLabelPoint[left](B){\tiny $E$}
\tkzLabelPoint[below](I){\tiny $\color{blue}(i)$}

\end{scope}

\begin{scope}[shift={(10,0)}]
\tkzDefPoint(0,-.2){A}
\tkzDefPoint(1.8,1.5){B}
\tkzDefPoint(5.2,1.5){C}
\tkzDefPoint(7,-.2){D}
\tkzDefPoint(-.5,-2){E}
\tkzDefPoint(7.5,-2){F}
\tkzDefPoint(3.5,1.2){H}

\tkzDefPoint(3.5,-2){I}

\draw[line width=0mm,fill=gray!40!white] (-.5,-2) -- (0,-.2) -- (3.5,1.2) -- (7,-.2) -- (7.5,-2) -- (-.5,-2);

\tkzDrawPoints[size=2](A,B,C,D,H);
\tkzDrawSegments[dashed](E,A A,H H,D D,F)
\tkzDrawSegments[](A,B C,D)
\tkzDrawSegments[thick](B,C)

\tkzLabelPoint[below](H){\tiny $C_{i+1}$}
\tkzLabelPoint[right](C){\tiny $C_{i+2}$}
\tkzLabelPoint[left](B){\tiny $C_i$}

\tkzLabelPoint[below](I){\tiny $\color{blue}(ii)$}

\end{scope}

\begin{scope}[shift={(20,0)}]
\tkzDefPoint(-.5,-2){A}
\tkzDefPoint(.2,0){B}
\tkzDefPoint(2,1.5){C}
\tkzDefPoint(3.5,1.04){D}
\tkzDefPoint(5,1.5){E}
\tkzDefPoint(6.8,0){F}
\tkzDefPoint(7.5,-2){G}
\tkzDefPoint(3.5,-2){I}

\draw[line width=0mm,fill=gray!40!white] (-.5,-2) -- (0.2,0) -- (3.5,1.04) -- (6.8,0) -- (7.5,-2) -- (-.5,-2);


\tkzDrawPoints[size=2](B,C,E,F,D);
\tkzDrawSegments[dashed](A,B B,E C,F F,G)
\tkzDrawSegments[](B,C E,F)
\tkzDrawSegments[thick](C,E)

\tkzLabelPoint[below](D){\tiny $R$}
\tkzLabelPoint[left](B){\tiny $C_{i-1}$}
\tkzLabelPoint[left](C){\tiny $C_{i}$}
\tkzLabelPoint[right](E){\tiny $C_{i+1}$}
\tkzLabelPoint[right](F){\tiny $C_{i+2}$}

\tkzLabelPoint[below](I){\tiny $\color{blue}(iii)$}

\end{scope}

\end{tikzpicture}
\end{center}

\begin{enumerate}
    \item The ray of
${\rm Nef}({\rm Bl}_e\mathbb P)$
orthogonal to the facet ${\rm Cone}( E,C_0) =
{\rm Cone}( E,qH_2 - E)$ is generated by $H_2$. Since the latter
lies in the relative interior of the facet and belongs to the tuple $\mathcal S$,
it lies in 
${\rm Cone}(\mathcal S \setminus E)$ and 
and in ${\rm Cone}(\mathcal S \setminus qH_2-E)$
respectively, so that~\eqref{eq:cond} holds.
We can reason in the same way for the facet 
${\rm Cone}( C_{n},E ) =
{\rm Cone}( qH_1 - E,E )$.
\item By Remark~\ref{rem:ort}, the ray of ${\rm Nef}({\rm Bl}_e\mathbb P)$ orthogonal to the facet ${\rm Cone}( C_i,C_{i+2} )$ 
is generated by $C_{i+1}$. Since the latter belongs to $\mathcal S$, 
condition~\eqref{eq:cond} holds.
\item Finally, let us 
consider a facet 
${\rm Cone}( C_i,C_{i+1} )$. Again 
by Remark~\ref{rem:ort} we have that ${\rm Cone}( C_{i-1},C_{i+1})$ 
is orthogonal to $C_i$, while ${\rm Cone}( C_{i},C_{i+2})$ is orthogonal to
$C_{i+1}$. Therefore, the ray orthogonal to the facet 
${\rm Cone}( C_i,C_{i+1} )$ is the intersection 
${\rm Cone}( C_{i-1},C_{i+1}) \cap {\rm Cone}( C_{i},C_{i+2})$.
We conclude that~\eqref{eq:cond} holds also for this facet.
\end{enumerate}

\end{proof}

\begin{remark}
\label{rem:m1}
The above proof relies on the results of Lemma~\ref{lem:mah}.
As we already pointed out in Remark~\ref{rem:gen}, this
allows to show that the Cox ring of 
${\rm Bl}_e\mathbb P$ is finitely generated,
but it does not give any clue on its generators. 
Therefore, even if we know that the effective cone is generated
in multiplicity $1$, we can not deduce that the same holds for
the Cox ring.
Nevertheless, since the first row of
the matrix~\eqref{eq:mat} corresponds
to the monomial $x_1x_2^p-x_3x_4^p$
in the lattice ideal $I_{\mathbb P}$, we can conclude that for any three 
indexes $i,j,k\in\{1,2,3,4\}$, 
$I_{\mathbb P}
\not \subseteq \langle x_i,x_j,x_k
\rangle^2$. Therefore we have the 
following.
\end{remark}
\begin{corollary}
\label{cor:m1}
Let $\mathbb P$ be a minimal toric surface
of Picard rank two.
If Conjecture~\ref{con:m1}
holds, then the Cox ring 
${\rm Cox}({\rm Bl}_e\mathbb P)$
is generated in multiplicity $1$.
\end{corollary}

\section{A Magma library}
\label{sec:ex}
In this section we present a library
of Magma functions~\cite{mag} that we used in
this paper and we explain some of 
them in one example.
The library is freely downloadable
from this web site:
\begin{center}
\url{https://github.com/alaface/Blowing-up-toric-surfaces}
\end{center}

The main functions of this library are briefly
described here.

\begin{itemize}  
  \item {\tt TestMultOne}. It takes as input
  a list of primitive lattice vectors of $\mathbb Z^2$ generating a fan $\Sigma$ of
  a complete toric surface $\mathbb P$. 
  It returns {\tt true} if the Cox ring of 
  ${\rm Bl}_e\mathbb P$ is generated in multiplicity one.
  
  \item {\tt TestComp}. It takes as input
  a list of primitive lattice vectors of $\mathbb Z^2$ generating a fan $\Sigma$ of
  a complete toric surface $\mathbb P$. 
  It returns {\tt true} if the lattice ideal of
  $\mathbb P$ is not contained in any square of ideals generated by subsets of three variables.
    
  \item {\tt GensUpTo}. It takes as input
  a list of primitive lattice vectors of $\mathbb Z^2$ generating a fan $\Sigma$ of
  a complete toric surface $\mathbb P$ and an integer $m$.
  It returns the
  divisor classes of generators of the Cox ring of ${\rm Bl}_e\mathbb P$ up to  multiplicity $m$.
      
  \item {\tt IsMDS}. It takes as input
  a list of primitive lattice vectors of $\mathbb Z^2$ generating a fan $\Sigma$ of
  a complete toric surface $\mathbb P$ and an integer $m$.
  It returns {\tt true} if the Cox ring of 
  ${\rm Bl}_e\mathbb P$ admits a pseudogenerating tuple in  multiplicity up to $m$.
\end{itemize}

\begin{example}
Let $\mathbb P := \mathbb P_{\Sigma}$
be the toric surface whose fan
is generated by all the primitive
lattice vectors of $\mathbb Z^2$ with
coordinates having absolute value at most $2$.
We display the rays in the following picture.

\begin{center}
\begin{tikzpicture}[scale=0.6]
  \tkzDefPoint(0,0){O}
  \foreach \x/\y/\name/\ind/\pos in {
  1/0/A1/1/right, 
  2/1/A2/2/right, 
  1/1/A3/3/above right, 
  1/2/A4/4/above right, 
  0/1/A5/5/above, 
  -1/2/A6/6/above left, 
  -1/1/A7/7/above left, 
  -2/1/A8/8/left, 
  -1/0/A9/9/left, 
  -2/-1/A10/10/left, 
  -1/-1/A11/11/below left, 
  -1/-2/A12/12/below left, 
  0/-1/A13/13/below, 
  1/-2/A14/14/below right, 
  1/-1/A15/15/below right,
  2/-1/A16/16/right}{
    \tkzDefPoint(\x,\y){\name}
    \tkzLabelPoint[\pos](\name){\tiny $x_{\ind}$}
    \tkzDrawSegment (O,\name)
    \tkzDrawPoint[size=2](\name)
    }
\draw[densely dotted] (-2.2,-2.2) grid (2.2,2.2);
\end{tikzpicture}
\end{center}
Using {\tt TestComp} we found that 
all the triples
$\{i,j,k\}$ such that $I_\mathbb P\subseteq
\langle x_i,x_j,x_k\rangle^2$ are
$\{2,6,12\}$, $\{2,7,12\}$,
$\{2,8,12\}$, $\{4,10,14\}$, $\{4,10,15\}$,
$\{4,10,16\}$, $\{6,10,16\}$, $\{6,11,16\}$,
$\{6,12,16\}$, $\{2,8,14\}$, $\{3,8,14\}$,
$\{4,8,14\}$.
The maximal subsets of rays which do not 
contain these triples are the following
\[
\begin{array}{ll}
    \{ 1, 3, 4, 5, 7, 9, 11, 12, 13, 14, 15, 16 \}&
    \{ 1, 2, 3, 4, 5, 6, 7, 9, 13, 14, 15, 16 \}\\
    \{ 1, 3, 5, 7, 8, 9, 10, 11, 12, 13, 15, 16 \}&
    \{ 1, 2, 3, 4, 5, 7, 9, 11, 13, 14, 15, 16 \}\\
    \{ 1, 3, 5, 7, 9, 10, 11, 12, 13, 14, 15, 16 \}&
    \{ 1, 2, 3, 4, 5, 7, 8, 9, 11, 13, 15, 16 \}\\
    \{ 1, 2, 3, 5, 6, 7, 8, 9, 10, 11, 13, 15 \}&
    \{ 1, 2, 3, 5, 6, 7, 9, 10, 11, 13, 14, 15 \}\\
    \{ 1, 5, 6, 7, 8, 9, 10, 11, 12, 13, 14, 15 \}&
    \{ 1, 2, 3, 5, 9, 10, 11, 12, 13, 14, 15, 16 \}\\
    \{ 1, 3, 4, 5, 7, 8, 9, 11, 12, 13, 15, 16 \}&
    \{ 1, 2, 3, 4, 5, 9, 10, 11, 12, 13 \}\\
    \{ 1, 3, 4, 5, 6, 7, 8, 9, 11, 12, 13, 15 \}&
    \{ 1, 2, 3, 4, 5, 6, 7, 8, 9, 13, 15, 16 \}\\
    \{ 1, 5, 7, 8, 9, 10, 11, 12, 13, 14, 15, 16 \}&
    \{ 1, 5, 6, 7, 8, 9, 13, 14, 15, 16 \}\\
    \{ 1, 3, 5, 6, 7, 9, 10, 11, 12, 13, 14, 15 \}&
    \{ 1, 3, 4, 5, 6, 7, 8, 9, 10, 11, 12, 13 \}\\
    \{ 1, 2, 3, 4, 5, 9, 11, 12, 13, 14, 15, 16 \}&
    \{ 1, 3, 5, 6, 7, 8, 9, 10, 11, 12, 13, 15 \}\\
    \{ 1, 2, 3, 4, 5, 6, 7, 8, 9, 10, 11, 13 \}&
    \{ 1, 2, 3, 5, 7, 9, 10, 11, 13, 14, 15, 16 \}\\
    \{ 1, 2, 3, 5, 7, 8, 9, 10, 11, 13, 15, 16 \}&
    \{ 1, 2, 3, 4, 5, 6, 7, 9, 11, 13, 14, 15 \}\\
    \{ 1, 2, 3, 4, 5, 6, 7, 8, 9, 11, 13, 15 \}&
    \{ 1, 3, 4, 5, 6, 7, 9, 11, 12, 13, 14, 15 \}
\end{array}
\]

Using {\tt TestMultOne} we proved that 
the Cox ring of ${\rm Bl}_e\mathbb P'$
is generated in multiplicity one for each
toric surface $\mathbb P'$ whose rays
of the fan are given by one of the above subsets.
The same
conclusion holds for any toric surface
dominated by any such $\mathbb P'$, i.e. whose fan
consists of a subset of those rays.
In particular Conjecture~\ref{con:m1} holds
for all these toric surfaces.
Using {\tt GensUpTo} we found four generators
in multiplicity $2$ for the Cox ring of 
${\rm Bl}_e\mathbb P'$, where $\mathbb P'$
is the toric surface defined by the rays
$\{2,4,6,8,10,12,14,16\}$. Finally, using
{\tt IsMDS} we checked that the Cox ring of
${\rm Bl}_e\mathbb P'$ is finitely generated.
\end{example}

\bibliographystyle{plain}
\bibliography{ref.bib}

\begin{bibdiv}
\begin{biblist}

\bib{adhl}{book}{
      author={Arzhantsev, Ivan},
      author={Derenthal, Ulrich},
      author={Hausen, J\"{u}rgen},
      author={Laface, Antonio},
       title={Cox rings},
      series={Cambridge Studies in Advanced Mathematics},
   publisher={Cambridge University Press, Cambridge},
        date={2015},
      volume={144},
        ISBN={978-1-107-02462-5},
      review={\MR{3307753}},
}

\bib{mag}{incollection}{
      author={Bosma, Wieb},
      author={Cannon, John},
      author={Playoust, Catherine},
       title={The {M}agma algebra system. {I}. {T}he user language},
        date={1997},
      volume={24},
       pages={235\ndash 265},
         url={https://doi.org/10.1006/jsco.1996.0125},
      review={\MR{1484478}},
}

\bib{cltu}{article}{
      author={Castravet, Ana-Maria},
      author={Laface, Antonio},
      author={Tevelev, Jenia},
      author={Ugaglia, Luca},
       title={Blown-up toric surfaces with non-polyhedral effective cone},
        date={2023},
     journal={Journal für die reine und angewandte Mathematik (Crelles
  Journal)},
      volume={2023},
      number={800},
       pages={1\ndash 44},
         url={https://doi.org/10.1515/crelle-2023-0022},
}

\bib{ck}{article}{
      author={Cutkosky, Steven~Dale},
      author={Kurano, Kazuhiko},
       title={Asymptotic regularity of powers of ideals of points in a weighted
  projective plane},
        date={2011},
        ISSN={2156-2261,2154-3321},
     journal={Kyoto J. Math.},
      volume={51},
      number={1},
       pages={25\ndash 45},
         url={https://doi.org/10.1215/0023608X-2010-019},
      review={\MR{2784746}},
}

\bib{DMN}{article}{
      author={Draisma, Jan},
      author={McAllister, Tyrrell~B.},
      author={Nill, Benjamin},
       title={Lattice-width directions and {M}inkowski's {$3^d$}-theorem},
        date={2012},
        ISSN={0895-4801},
     journal={SIAM J. Discrete Math.},
      volume={26},
      number={3},
       pages={1104\ndash 1107},
         url={https://doi.org/10.1137/120877635},
      review={\MR{3022127}},
}

\bib{E}{book}{
      author={Eisenbud, David},
       title={The geometry of syzygies},
      series={Graduate Texts in Mathematics},
   publisher={Springer-Verlag, New York},
        date={2005},
      volume={229},
        ISBN={0-387-22215-4},
        note={A second course in commutative algebra and algebraic geometry},
      review={\MR{2103875}},
}

\bib{gk}{article}{
      author={Gonz\'{a}lez, Jos\'{e}~Luis},
      author={Karu, Kalle},
       title={Some non-finitely generated {C}ox rings},
        date={2016},
        ISSN={0010-437X,1570-5846},
     journal={Compos. Math.},
      volume={152},
      number={5},
       pages={984\ndash 996},
         url={https://doi.org/10.1112/S0010437X15007745},
      review={\MR{3505645}},
}

\bib{ggk2019}{article}{
      author={Gonz\'{a}lez~Anaya, Javier},
      author={Gonz\'{a}lez, Jos\'{e}~Luis},
      author={Karu, Kalle},
       title={Constructing non-{M}ori dream spaces from negative curves},
        date={2019},
        ISSN={0021-8693,1090-266X},
     journal={J. Algebra},
      volume={539},
       pages={118\ndash 137},
         url={https://doi.org/10.1016/j.jalgebra.2019.08.005},
      review={\MR{3995238}},
}

\bib{ggkneg}{article}{
      author={Gonz\'{a}lez~Anaya, Javier},
      author={Gonz\'{a}lez, Jos\'{e}~Luis},
      author={Karu, Kalle},
       title={On a family of negative curves},
        date={2019},
        ISSN={0022-4049,1873-1376},
     journal={J. Pure Appl. Algebra},
      volume={223},
      number={11},
       pages={4871\ndash 4887},
         url={https://doi.org/10.1016/j.jpaa.2019.02.019},
      review={\MR{3955045}},
}

\bib{ggk2021}{article}{
      author={Gonz\'{a}lez-Anaya, Javier},
      author={Gonz\'{a}lez, Jos\'{e}~Luis},
      author={Karu, Kalle},
       title={Curves generating extremal rays in blowups of weighted projective
  planes},
        date={2021},
        ISSN={0024-6107,1469-7750},
     journal={J. Lond. Math. Soc. (2)},
      volume={104},
      number={3},
       pages={1342\ndash 1362},
         url={https://doi.org/10.1112/jlms.12461},
      review={\MR{4332479}},
}

\bib{ggkopen}{article}{
      author={Gonz\'{a}lez~Anaya, Javier},
      author={Gonz\'{a}lez, Jos\'{e}~Luis},
      author={Karu, Kalle},
       title={Nonexistence of negative curves},
        date={2023},
        ISSN={1073-7928,1687-0247},
     journal={Int. Math. Res. Not. IMRN},
      number={16},
       pages={14368\ndash 14400},
         url={https://doi.org/10.1093/imrn/rnac355},
      review={\MR{4631435}},
}

\bib{hkl2}{article}{
      author={Hausen, Juergen},
      author={Keicher, Simon},
      author={Laface, Antonio},
       title={On blowing up the weighted projective plane},
        date={2016},
     journal={Mathematische Zeitschrift},
      volume={290},
       pages={1339\ndash 1358},
         url={https://api.semanticscholar.org/CorpusID:119133636},
}

\bib{hkl}{article}{
      author={Hausen, J\"{u}rgen},
      author={Keicher, Simon},
      author={Laface, Antonio},
       title={Computing {C}ox rings},
        date={2016},
        ISSN={0025-5718,1088-6842},
     journal={Math. Comp.},
      volume={85},
      number={297},
       pages={467\ndash 502},
         url={https://doi.org/10.1090/mcom/2989},
      review={\MR{3404458}},
}

\bib{Kr}{article}{
      author={Kraft, J\"{u}rgen},
       title={Singularity of monomial curves in {${\bf A}^3$} and {G}orenstein
  monomial curves in {${\bf A}^4$}},
        date={1985},
        ISSN={0008-414X,1496-4279},
     journal={Canad. J. Math.},
      volume={37},
      number={5},
       pages={872\ndash 892},
         url={https://doi.org/10.4153/CJM-1985-047-8},
      review={\MR{806645}},
}

\bib{lu}{article}{
      author={Laface, Antonio},
      author={Ugaglia, Luca},
       title={Effective cone of the blow-up of the symmetric product of a
  curve},
        date={2022},
     journal={arXiv:2210.11829},
}

\bib{la}{book}{
      author={Lazarsfeld, Robert},
       title={Positivity in algebraic geometry. {I}},
      series={Ergebnisse der Mathematik und ihrer Grenzgebiete. 3. Folge. A
  Series of Modern Surveys in Mathematics [Results in Mathematics and Related
  Areas. 3rd Series. A Series of Modern Surveys in Mathematics]},
   publisher={Springer-Verlag, Berlin},
        date={2004},
      volume={48},
        ISBN={3-540-22533-1},
         url={https://doi.org/10.1007/978-3-642-18808-4},
        note={Classical setting: line bundles and linear series},
      review={\MR{2095471}},
}

\bib{ms}{book}{
      author={Miller, Ezra},
      author={Sturmfels, Bernd},
       title={Combinatorial commutative algebra},
      series={Graduate Texts in Mathematics},
   publisher={Springer-Verlag, New York},
        date={2005},
      volume={227},
        ISBN={0-387-22356-8},
      review={\MR{2110098}},
}

\bib{na}{article}{
      author={Nakamaye, Michael},
       title={Stable base loci of linear series},
        date={2000},
        ISSN={0025-5831,1432-1807},
     journal={Math. Ann.},
      volume={318},
      number={4},
       pages={837\ndash 847},
         url={https://doi.org/10.1007/s002080000149},
      review={\MR{1802513}},
}

\bib{PS}{article}{
      author={Peeva, Irena},
      author={Sturmfels, Bernd},
       title={Syzygies of codimension {$2$} lattice ideals},
        date={1998},
        ISSN={0025-5874},
     journal={Math. Z.},
      volume={229},
      number={1},
       pages={163\ndash 194},
         url={https://doi.org/10.1007/PL00004645},
      review={\MR{1649322}},
}

\bib{Pe}{article}{
      author={Peth\H{o}, Attila},
       title={Simple continued fractions for the {F}redholm numbers},
        date={1982},
        ISSN={0022-314X},
     journal={J. Number Theory},
      volume={14},
      number={2},
       pages={232\ndash 236},
         url={https://doi.org/10.1016/0022-314X(82)90048-8},
      review={\MR{655727}},
}

\bib{Ra}{article}{
      author={Rabinowitz, Stanley},
       title={A census of convex lattice polygons with at most one interior
  lattice point},
        date={1989},
        ISSN={0381-7032},
     journal={Ars Combin.},
      volume={28},
       pages={83\ndash 96},
      review={\MR{1039134}},
}

\bib{rg}{book}{
      author={Rosales, J.~C.},
      author={Garc\'{\i}a-S\'{a}nchez, P.~A.},
       title={Finitely generated commutative monoids},
   publisher={Nova Science Publishers, Inc., Commack, NY},
        date={1999},
        ISBN={1-56072-670-9},
      review={\MR{1694173}},
}

\bib{V}{book}{
      author={Vajda, S.},
       title={Fibonacci \& {L}ucas numbers, and the golden section},
      series={Ellis Horwood Series: Mathematics and its Applications},
   publisher={Ellis Horwood Ltd., Chichester; Halsted Press [John Wiley \&
  Sons, Inc.], New York},
        date={1989},
        ISBN={0-470-21508-9},
        note={Theory and applications, With chapter XII by B. W. Conolly},
      review={\MR{1015938}},
}

\end{biblist}
\end{bibdiv}

\end{document}